\documentclass[a4paper,12pt,reqno]{amsart}
\usepackage[latin1]{inputenc}
\usepackage{amssymb}
\usepackage{latexsym}
\usepackage{amsfonts}
\usepackage{amsmath}
\usepackage{bm}
\usepackage{bbm}
\usepackage{mathrsfs}
\usepackage[T1]{fontenc}
\usepackage{ae}
\usepackage[english]{varioref}
\usepackage{graphicx}
\usepackage{a4wide}
\usepackage{enumitem}
\usepackage[draft, english]{fixme}
\usepackage{amsthm}
\usepackage{xcolor}
\newcommand{\N}{\mathbbm{N}}
\newcommand{\Z}{\mathbbm{Z}}

\newcommand{\R}{\mathbbm{R}}
\newcommand{\C}{\mathbbm{C}}

\newcommand{\T}{\mathbbm{T}}
\newcommand{\een}{\mathbbm{1}}

\theoremstyle{plain}
\newtheorem{thm}{Theorem}[section] 
\newtheorem*{thm*}{Theorem} 

\newtheorem{lemma}[thm]{Lemma}
\newtheorem{cor}[thm]{Corollary}
\newtheorem*{cond1*}{Condition $D(u_n)$}
\newtheorem*{cond2*}{Condition $D'(u_n)$}
\newtheorem*{cond3*}{Condition $D_{g(r)}'(u_n)$}
\newtheorem*{condC*}{Condition $C(u_n)$}

\theoremstyle{definition}
\newtheorem{defn}[thm]{Definition}

\theoremstyle{remark}

\newtheorem{remark}[thm]{Remark}
\newtheorem*{remark*}{Remark}
\newtheorem*{claim*}{Claim}

\newcommand{\tub}[1]{\left \{#1\right \}}

\newcommand{\dr}{\textup{SL}(d,\R)}
\newcommand{\dz}{\textup{SL}(d,\Z)}
\newcommand{\drdz}{\dr/\dz}

\newcommand{\iir}{\textup{SL}(2,\R)}
\newcommand{\iiz}{\textup{SL}(2,\Z)}
\newcommand{\iiriiz}{\iir/\iiz}

\newcommand{\dd}{\textup{\textbf{d}}}

\newcommand{\PP}{\textup{\textbf{P}}}
\newcommand{\gb}{\bar{g}}
\newcommand{\ob}{\bar{\omega}}

\newcommand{\bis}{B_{\frac1s}(x)}
\newcommand{\ssp}{\,:\,}
\newcommand{\para}[1]{\left(#1\right)}
\newcommand{\kpara}[1]{\left[#1\right]}
\newcommand{\num}[1]{\left \vert #1\right \vert}

\newcommand{\tyk}{\mathcal}

\newcommand{\supp}{\textup{supp}}
\newcommand{\norm}[1]{\left\Vert #1\right\Vert}

\newcommand{\TX}{\tyk{X}}
\newcommand{\TY}{\tyk{Y}}

\newcommand{\exist}{\exists\,}
\newcommand{\foral}{\forall\,}

\newcommand{\thelam}{\theta_{\lambda}}

\renewcommand{\phi}{\varphi}
\renewcommand{\epsilon}{\varepsilon}

\numberwithin{equation}{section}

\makeatletter \newcommand\@b@gtimes[1]{%
\vcenter{\hbox{#1$\m@th\mkern-2mu\times\mkern-2mu$}}} \newcommand\@bigtimes{%
\mathchoice{\@b@gtimes\huge} {\@b@gtimes\LARGE}
{\@b@gtimes{}} {\@b@gtimes\footnotesize} 
} \newcommand\bigtimes{\mathop{\@bigtimes}\displaylimits} \makeatother

\title{Extreme value theory for random walks on homogeneous spaces}
\author{Maxim S\o lund Kirsebom}\thanks{This research was supported by ERC grant 239606}
\address{School of Mathematics, University of Bristol, Bristol, U.K.}
\email{maxim.kirsebom@bristol.ac.uk}

\begin{document}

\maketitle

%
%

\subsection*{Abstract}
In this paper we study extreme events for random walks on homogeneous spaces. We consider the following three cases. On the torus we study closest returns of a random walk to a fixed point in the space. For a random walk on the space of unimodular lattices we study extreme values for lengths of the shortest vector in a lattice. For a random walk on a homogeneous space we study the maximal distance a random walk gets away from an arbitrary fixed point in the space. We prove an exact limiting distribution on the torus and upper and lower bounds for sparse subsequences of random walks in the two other cases. In all three settings we obtain a logarithm law. 

%
%

\section{Introduction}

Let $\mathcal{X}$ be a probability space and $G$ a group acting on $\TX$. Let $m$ be a $G$-invariant probability measure on $\TX$ and fix also a probability measure $\mu$ on $G$. We define a random walk on $\TX$ as a sequence of random variables $X_i=g_i\cdots g_1x$ where the $g_j$'s have distribution $\mu$ and $x$ has distribution $m$. Fix a function $\Delta:\TX\to\R$. The focus of our interest is the random variable 
\begin{align*}
M_n=\max_{0\leq i< n} \Delta(X_i).
\end{align*}
There exists a natural measure on the space of all random walks on $X$ which we denote by $\PP$ and define formally in section \ref{NotAndSet}.
We are particularly interested in the existence of sequences $a_n$ and $b_n$ such that the distribution $\PP(M_n\leq a_nr +b_n)$ has a non-degenerate limit and if this is the case, determining the limit. We refer to such a limit as the extreme value distribution of the random walk.
One reason why extreme value distributions are interesting is that they imply asymptotics for the growth of extreme values of $\Delta(X_n)$. In many cases this turns out to be a logarithm law, namely we get that almost surely
\begin{align*}
\limsup_{n\to\infty}\frac{\Delta(X_n)}{\log n}=C
\end{align*}
for some $C>0$.
One result of this kind is Sullivans logarithm law for geodesics on hyperbolic $d$-space \cite{Sull}. Kleinbock and Margulis later generalised this to certain classes of homogeneous spaces, see \cite{KleinMarg}, and Athreya, Ghosh and Prasad proved ultrametric analogues of this result, see \cite{AGP1}, \cite{AGP2}. 

The general framework for determining extreme value distributions is known as extreme value theory (EVT). EVT was first applied in dynamics by Collet \cite{Collet}, who studied $C^2$ transformations $T$ of an interval. He was interested in the entrance times of $T^j x$ into a shrinking neighborhood around a fixed point $x_0$ and to understand this, he determined the limiting distribution of the maximum of $-\log\dd(T^j x,x_0)$. Similar results to Collets have since been proven for other choices of $T$ and other types of maps, see for example \cite{FF}, \cite{Gupta1}, \cite{GHN}, \cite{HNT}. In the context of this paper, recent results by Aytac, Freitas and Vaienti \cite{AFV} are particularly interesting as they apply EVT to a setting involving randomness, more precisely, iterations of a randomly perturbed map.
Freitas, Freitas and Todd have developed a general framework for applying EVT to dynamical systems $T:X\to X$, see \cite{FFT1}, \cite{FFT2}, \cite{FFT3}. 

Classically, random walks were studied as objects living on $\R^d$. However, the concept of random walks generalizes easily to many other spaces, for example to homogeneous spaces with a group action which we are particularly interested in.
In \cite{EskMarg}, Eskin and Margulis studied recurrence properties for random walks on finite volume homogeneous spaces $G/\Gamma$ where $G$ is a semisimple Lie group and $\Gamma$ a nonuniform irreducible lattice. In a series of papers Benoist and Quint  \cite{BQ3}, \cite{BQ1}, \cite{BQ2}, \cite{BQ4}, developed this theory further by studying stationary measures on $G/\Gamma$ while also generalizing their results to $p$-adic Lie groups. 

The main idea of this paper is to apply EVT to random walks on homogeneous spaces. The level of dependency among the $X_i$'s is the deciding factor in whether EVT can successfully be applied to obtain limiting distributions for the maximum of $\Delta(X_i)$. The closer $X_i$ is to being an independent sequence the easier it is to apply EVT. EVT provides independence-like conditions that, if satisfied by $X_i$, imply a limiting distribution for $M_n$. The idea of this paper is to verify these conditions by rewriting the joint distribution of the random walk using the averaging operator. The spectral gap property of the averaging operator is the crucial ingredient in showing that the independence-like conditions are satisfied by the random walk.

Our main results are divided into three different settings. In the following, let $S_{\mu}$ and $G_{\mu}$ denote the semigroup and group generated by the support of $\mu$ respectively. 
\subsection{Closest returns on the torus}
\label{SecRes1}

Let $\TX=\T^d$ be the $d$-dimensional torus with Lebesque measure $m$ and Euclidian metric $\dd$. Let $G=\textup{Aut}(\T^d)$ denote the group of linear automorphisms of $\T^d$ and fix a probability measure $\mu$ on the group. We assume that there is no $G_{\mu}$-invariant factor torus $\overline{T}$ of $\T^d$ such that the projection of $G_{\mu}$ on $\textup{Aut}(\overline{T})$ is amenable.

We are interested in the closest returns of a random walk to a fixed point on the torus and in particular, how these shortest distances distribute. Let $x_0\in \TX$ be fixed and define
\begin{align*}
\Delta(x)=-\log\dd(x,x_0).
\end{align*}
We see that for small values of $\dd(x,x_0)$, $\Delta(x)$ becomes large hence we can study the closest returns of $X_i$ by looking at successive maxima of $\Delta(X_i)$.
\begin{thm}[]
\label{thm:[Torus]}
Assume that the support of $\mu$ is bounded and that $\det(g-\textup{I})\neq 0$ for all $g\in S_{\mu}$. Then for $u_n=r+\frac1d \log n$ we have that for a.e. $x_0\in \TX$
\begin{align*}
\lim_{n\to\infty} \PP(M_n\leq u_n)=e^{-\frac{1}{V_d}e^{-dr}},
\end{align*} 
where $V_d$ is the volume of the unit ball in $\R^d$.
\end{thm}
The stationary measures of random walks on the torus were recently studied by Bourgain, Furman, Lindenstrauss and Mozer \cite{BFLM}.

The limiting distribution implies a logarithm law. That is,
\begin{cor}[]
\label{cor:[Torus]}
For $\PP$-a.e. random walk and every $x_0\in \TX$ we have
\begin{align*}
\limsup_{n\to\infty}\frac{\Delta(X_n)}{\log n}=\frac{1}{d}.
\end{align*}
\end{cor}
Actually we will see later that we only need a sufficiently good upper bound on the limiting distribution of $M_n$ to derive the logarithm law.

%
%
%
%
\subsection{Shortest vectors on the space of unimodular lattices}
\label{SecRes2}

Let $\TX=\tyk{L}_d$ denote the space of $d$-dimensional unimodular lattices and let $m$ denote the normalized Haar measure on $\tyk{L}_d$. Recall that $\tyk{L}_d$ can be identified with $\drdz$ and thus can be thought of as a homogeneous space. Let $G=\dr$ and fix a probability measure $\mu$ on the group. Assume that $\overline{G_{\mu}}$ is non-amenable. Set
\begin{align}
\label{eq:ShortestVec}
\Delta(\Lambda)=\max_{v\in \Lambda\backslash\tub{0}}\log\para{\frac{1}{\norm{v}}}.
\end{align}
We see that this maximum will always be attained for the shortest vector in the lattice $\Lambda$. The function $\Delta$ plays a crucial role in connections between flows on the space of lattices in $\R^d$ and Diophantine approximation.

Define
\begin{align}
\label{eq:SparseMax}
M_{n,a}=\max_{0\leq i<n} \Delta(X_{ai})
\end{align}
\begin{thm}[]
\label{thm:[Lattice1]}
Set $u_n=r+\frac1d\log n$. Let $w=\frac{V_d}{2\zeta(d)}$ where $V_d$ is the volume of the unit ball in $\R^d$. There exist constants $w(a)\in \R$ such that $w(a)\to w$ as $a\to\infty$ and 
\begin{align*}
e^{-we^{-dr}}\leq \liminf_{n\to\infty}\PP(M_{n,a}\leq u_n)\leq \limsup_{n\to\infty} \PP(M_{n,a}\leq u_n)\leq e^{-w(a)e^{-dr}}.
\end{align*}
\end{thm}
As the reader will notice, we are not able to prove an exact limiting distribution. Instead we get an upper and lower bound only differing by a constant multiple which goes to zero as the random walk becomes infinitely sparse. The difference between this case and the random walk on the torus is that one of the independence-like conditions from EVT is not fully satisfied in this setup.
It is natural to ask what additional assumptions would suffice to prove an exact limit. This question is answered by the following theorem.
\begin{thm}[]
\label{thm:[Lattice2]}
Let $\tub{m_j}$ be a sequence in $\N$ such that $\tub{m_{j+1}-m_j}$ is strictly increasing. Also, let $\alpha_n<\beta_n$ denote sequences in $\N$ such that $\alpha_n\to\infty$ and $N_n:=\beta_n-\alpha_n\to\infty$. Then for $u_n=r+\frac1d\log N_n$ we have
\begin{align*}
\lim_{n\to\infty} \PP\para{\max_{\alpha_n\leq j<\beta_n}\Delta\para{X_{m_j}}\leq u_n}=e^{-we^{-dr}},
\end{align*}
where $w$ is the constant from Theorem \ref{thm:[Lattice1]}.
\end{thm}

Again we obtain a logarithm law.
\begin{cor}
\label{cor:[Lattice]}
For $\PP$-almost every random walk and every $x_0\in \TX$ we have
\begin{align*}
\limsup_{n\to\infty}\frac{\Delta(X_n)}{\log n}=\frac{1}{d}.
\end{align*}
\end{cor}

%
%
%
%

\subsection{Maximal excursions on homogeneous spaces}
\label{SecRes3}

Let $\TX=G/\Gamma$ where $G$ is a simple, non-compact Lie group with finite center and $\Gamma$ a non-uniform lattice in $G$. Let $m$ denote the normalized Haar measure on $\TX$ and fix also a probability measure $\mu$ on $G$. Assume that $\overline{G_{\mu}}$ is non-amenable.

We are interested in the maximal distance a random walk gets away from some arbitrary fixed point $x_0\in \TX$. Therefore, define
\begin{align*}
\Delta(x)=\dd(x,x_0)
\end{align*}
where $\dd$ is a Riemannian metric on $\TX$ chosen by fixing a right invariant Riemannian metric on $G$ which is bi-invariant with respect to a maximal compact subgroup of $G$.  Let $M_{n,a}$ be defined as in \eqref{eq:SparseMax}.
\begin{thm}[]
\label{thm:[Cusp]}
There exists constants $k>0$, $w>0$ and $w(a)\in\R$ such that for sufficiently large $a$ we have $w(a)>0$ and 
\begin{align*}
e^{-we^{-kr}}\leq\liminf_{n\to\infty} \PP(M_{n,a}\leq u_n)\leq \limsup_{n\to\infty} \PP(M_{n,a}\leq u_n)\leq e^{-w(a) e^{-kr}},
\end{align*}
where $u_n=r+\frac1k\log n$.
\end{thm}
\begin{remark}
\label{rem:k}
The constant $k$ is explicit and has been computed in \cite{KleinMarg} (Lemma 5.6).
\end{remark}
Again we do not obtain an exact limit and again this relates to the inability to verify one of the independence-like conditions from EVT. 
%
%
In this setting we also do not prove an analogue of Theorem \ref{thm:[Lattice2]}. The reason is that we need to know exact asymptotics for the tail distribution function of $\Delta$, a property which we call $k$-SDL (see definition \ref{DistanceLike}). While this was proven in \cite{KleinMarg} for the shortest vectors on $\tyk{L}_d$, only a weaker property called $k$-DL is known for the Riemannian distance on homogeneous spaces.

As in the previous cases a logarithm law follows from Theorem \ref{thm:[Cusp]}.
\begin{cor}[]
\label{cor:[Cusp]}
For $\PP$-almost every random walk and for all $x_0\in X$ we have
\begin{align*}
\limsup_{n\to\infty}\frac{\Delta(X_n)}{\log n}=\frac{1}{k}.
\end{align*}
Here $k>0$ is again the constant from Remark \ref{rem:k}.
\end{cor}
This is a random walk analogue of the logarithm law Kleinbock and Margulis proved for geodesics. A natural question to ask is whether one could determine the extreme value distribution for the geodesic flow, since it would be a generalization of the logarithm law mentioned. One result in this direction is by Pollicott \cite{MarkP}. He determines the exact limiting distribution for the geodesic flow on $\iiriiz$. However, the proof uses connections between geodesics on the upper half plane and continued fractions, a connection that only exists for $d=2$.

%
%
%
%
\subsection{Structure of the paper}

We begin by giving a short introduction to extreme value theory in Section \ref{EVT}. This introduction is short and in no way a complete overview of the field. However, for the reader unfamiliar with extreme value theory, the section should be sufficient to understand this paper without having to look elsewhere. In Section \ref{MainRes} we formally define the random walk and introduce the main tools used in the paper. In this section we also show how the averaging operator and its spectral gap property is used to prove quasi-independence for the random walk. We prove various results for the limiting distribution of the random walk under general assumptions. In Section \ref{IntroProofs} we finalize the proofs of our main theorems using the results from the previous section and known results. For the case of the torus an additional argument is required which we give in this section as well.

%
%

\section{General extreme value theory}
\label{EVT}

EVT deals with determining the distributional properties of the maximum or minimum of a sequence of random variables $f_n$ as $n$ becomes large. This task is fairly simple if one assumes that the random variables are mutually independent. However, in many interesting cases we have some degree of dependence among the random variables. How much we can prove in the dependent case is related to how strong the dependency among the random variables is.

In the following we elaborate on the basics of EVT for stationary sequences of identically distributed random variables. For a reference on general extreme value theory, see \cite{LGH}.

Let $(\TX,\PP)$ be a probability space. Let $f_i$ denote some stationary, identically distributed sequence of random variables and let $M_n:=\max_{0\leq i< n}(f_i)$. We use the notation
\begin{align*}
F_{f_0,\dots,f_{n-1}}(r)=\PP(f_0\leq r,\dots,f_{n-1}\leq r)=\PP(\tub{f_0\leq r}\cap\cdots\cap\tub{f_{n-1}\leq r}).
\end{align*}
Notice that $F_{f_0,\dots,f_{n-1}}(r)=\PP(M_n\leq r)$. Also notice that since the $f_i$ are identically distributed we have $F_{f_i}(r)=F_{f_j}(r)$ for all $i,j\in \N$. We denote this common distribution simply by $F$.
We are concerned with the limiting distribution of $M_n$ under linear scalings $a_n^{-1}(M_n-b_n)$, where $a_n>0$ and $b_n$ are sequences of real numbers. By this we mean the limit
\begin{align*}
\lim_{n\to\infty} \PP\para{\frac{M_n-b_n}{a_n}\leq r},
\end{align*} 
where $r\in\R$.
The sequences $a_n$ and $b_n$, known as scaling sequences, are introduced in order to avoid cases of \textit{degenerate} limiting distributions, a notion we explain in the following. To understand why degenerate cases occur, look for example at any i.i.d. stochastic process. In this case we easily see that
\begin{align*}
\PP(M_n\leq r)=F(r)^n \to
\begin{cases}
1 & \text{if }\, r\in\tub{r:F(r)=1}\\
0 & \text{if }\, r\in\tub{r:0\leq F(r)<1}
\end{cases}.
\end{align*}
We call this a degenerate limiting distribution and we see that such one provides us with little information about $M_n$. Later in this section we discuss how to determine $a_n$ and $b_n$, but for now assume these exist such that
\begin{align*}
\PP\para{\frac{M_n-b_n}{a_n}\leq r}=\PP\para{M_n\leq a_n r+b_n}\to G(r),
\end{align*}
where $G:\R\to [0,1]$ is a non-degenerate distribution function. To simplify notation set $u_n:=a_n r +b_n$.

As mentioned, the i.i.d. case is the simplest, and in this case the limiting distribution of $M_n$ is known. When dealing with the dependent case, we are interested in stationary sequences that only exhibit little dependency. In other words, these are sequences that in some sense are close to being independent. This notion is formalized through two independence type conditions denoted $D(u_n)$ and $D'(u_n)$.
\begin{cond1*}
Condition $D(u_n)$ will be said to hold for $f_i$ and $u_n$ if for any integers $0\leq i_1<\dots<i_p<j_1<\dots<j_{p'}< n$ for which $j_1-i_p\geq l$, we have
\begin{align*}
\left\vert F_{f_{i_1},\dots,f_{i_p},f_{j_1},\dots,f_{j_{p'}}}(u_n)-F_{f_{i_1},\dots,f_{i_p}}(u_n)F_{f_{j_1},\dots,f_{j_{p'}}}(u_n) \right\vert\leq \alpha(n,l),
\end{align*}
where there exists a sequence $l_n$ s.t. $\alpha(n,l_n)\to 0$ as $n\to\infty$ and $\frac{l_n}{n}\to 0$ for $n\to\infty$.
\end{cond1*}
\begin{cond2*}
Condition $D'(u_n)$ will be said to hold for $f_i$ and $u_n$ if
\begin{align*}
\limsup_{n\to\infty} n\sum_{j=1}^{\left[\frac{n}{q}\right]}\PP(f_0>u_n,f_j>u_n)\to 0\;\;\text{ as }\;\; q\to\infty.
\end{align*}
\end{cond2*}
It is a standard result from EVT that if a stationary sequence $f_i$ satisfies these two conditions, then the limiting distribution of $M_n$ is the same as if $f_i$ were an i.i.d. process. This is the content of the following theorem.
\begin{thm}[\cite{LGH}, Theorem 3.4.1.]
\label{thm:341orig}
Let $u_n=a_n r+b_n$ be a scaling sequence s.t. $D(u_n)$ and $D'(u_n)$ are satisfied for the stationary sequence $f_n$. If $\tau=\tau(r)$ is a real function such that $n\PP(f_0>u_n)\to\tau$,
then
\begin{align}
\label{eq:limdist}
\lim_{n\to\infty} \PP(M_n\leq u_n)=e^{-\tau}.
\end{align}
\end{thm}
For some cases of dependent stationary sequences, either or both of Condition $D(u_n)$ and $D'(u_n)$ are not satisfied. However, it is possible to weaken these conditions and still salvage some information about the limiting distribution. For the purpose of this paper we introduce the following weakened version of Condition $D'(u_n)$.
\begin{cond3*}
For the stationary sequence $f_i$, let
\begin{align*}
g_q(r):=\limsup_{n\to\infty} n\sum_{j=1}^{\left[\frac{n}{q}\right]}\PP(f_0>u_n,f_j>u_n).
\end{align*}
Condition $D_{g(r)}'(u_n)$ will be said to hold for $f_i$ and $u_n$ if
\begin{align*}
\limsup_{q\to\infty} g_q(r)\leq g(r),
\end{align*}
where $g:\R\to\R$ only depends on $r\in\R$.
\end{cond3*}

Under these weakened assumptions we can prove the following theorem.
\begin{thm}[]
\label{thm:341}
Let $u_n=a_n r+b_n$ be a scaling sequence s.t. $D(u_n)$ and $D_{g(r)}'(u_n)$ are satisfied for the stationary sequence $f_i$. If $\tau_1=\tau_1(r)$ and $\tau_2=\tau_2(r)$ denote real functions such that
\begin{align}
\label{eq:uplowbound2}
\tau_1\leq\liminf_{n\to\infty}n\PP(f_0>u_n)\leq\limsup_{n\to\infty}n\PP(f_0>u_n)\leq\tau_2.
\end{align}
Then
\begin{align*}
e^{-\tau_2}\leq\liminf_{n\to\infty} \PP(M_n\leq u_n)\leq \limsup_{n\to\infty} \PP(M_n\leq u_n)\leq e^{g(r)-\tau_1}.
\end{align*}
\end{thm}
The proof of Theorem \ref{thm:341} is essentially similar to the proof of Theorem \ref{thm:341orig}.
Notice that we also made a weakening of the assumption that $n\PP(f_0>u_n)\to\tau$. This is to accommodate cases where the limit cannot be determined or does not exist.

Until now we have assumed the existence of scaling sequences $a_n$ and $b_n$ such that the limit of $\PP(M_n\leq a_n r+b_n)$ is non-degenerate. However, such scaling sequences do not necessarily exist. In the case of Theorem \ref{thm:341orig}, the assumption that $n\PP(f_0>u_n)\to\tau$ provides the most straightforward way to determine if suitable $a_n$ and $b_n$ exist and, if this is the case, what they are. Namely, we see that if the limit function $\tau$ is either zero or infinity, then the limit in \eqref{eq:limdist} becomes a degenerate distribution. Thus in order to obtain a non-degenerate limit, we must choose $a_n$ and $b_n$ such that the limit $n\PP(f_0>u_n)\to\tau$ is non-trivial. In specific cases  writing out the expression for $n\PP(f_0>u_n)$ often provides an easy way to see how $a_n$ and $b_n$ must be chosen in order for the limit to exist and be non-trivial. 

Similarly for Theorem \ref{thm:341}, if $\tau_2=\infty$ then we get a trivial lower bound on the $\liminf$ of $\PP(M_n\leq u_n)$. So again, by looking at the expression for $n\PP(f_0>u_n)$ we can often see how $a_n$ and $b_n$ must be chosen for the upper bound on the $\limsup$ to be less than infinity.

%
%

\section{EVT for random walks in a general setting}
\label{MainRes}
In this section we define random walks on a general probability space with a group action. In this general setting we show how the averaging operator can be used to prove extreme value distributions and logarithm laws for random walks. First we introduce the setup and define notation.

\subsection{Notation and setup}
\label{NotAndSet}
Let $(\TX,m)$ denote a probability space and $G$ a group acting measurably on $\TX$ preserving $m$. Fix also some probability measure $\mu$ on $G$. The product space $G^{\times \N}$ naturally inherits the product measure $\mu^{\otimes \N}$ which is also a probability measure.

We define the probability space $(\TY,\PP)$ by
\begin{align*}
(\TY,\PP):=(G^{\times \N} \times \TX\,,\, \mu^{\otimes\N}\otimes m).
\end{align*}
We denote elements in $G^{\times\N}$ by $\gb$ and write these as $\gb=(g_1,\dots,g_i,\dots)$.

By a random walk on $\TX$ generated by $G$ we mean a sequence of the form $X_i=g_i\cdots g_1 x$ where $x\in \TX$ has distribution $m$ and the $g_j\in G$ have distribution $\mu$. Define the map $L^i:G^{\times \N} \to G$ by $L^i(\gb)=g_i\cdots g_1$. Then for each $i$, $L^i(\gb)x$ represents the $i$'th position of the random walk along the path $\gb$ starting at $x$. We use the convention that $L^0(\gb)=e$, i.e. the neutral element in $G$. We define the sequence of random variables $X_i:\TY\to \TX$ by
\begin{align*}
X_i(\gb,x)=L^i(\gb)x.
\end{align*}
We see that $\TY$ can be thought of as the space of all possible random walks on $\TX$. 
Let $\Delta: \TX\to\R$. We define the sequence of real random variables $\xi_i: \TY\to \R$ by
\begin{align*}
\xi_i(\gb,x)=\Delta(L^i(\gb)x).
\end{align*}
and define a new sequence of random variables $M_n:\TY\to\R$ by 
\begin{align*}
M_n(\gb,x):=\max_{0\leq i<n}\xi_i(\gb,x).
\end{align*}
It follows from $G$-invariance of $m$ that $\xi_i$ is a stationary sequence with respect to $\PP$. Stationarity in particular implies that the random variables are identically distributed and we let $F(r)$ denote the common distribution function of the $\xi_i$.

We denote by
\begin{align*}
G^i=G\cdots G=\tub{g_1\cdots g_i\;:\;g_j\in G, \;\forall\; 1\leq j\leq i}.
\end{align*}
The natural probability measure on $G^i$ is the convolution measure defined as the push-forward measure of $\mu^{\otimes i}$ under the map $L^i$. That is $\mu^{*i}=\mu^{\otimes \N}((L^i)^{-1})$. It is a useful observation that
\begin{align*}
\int_{G^{\times\N}} f(L^i(\gb))\,d\mu^{\otimes \N}(\gb)=\int_{G^i} f(g)\,d\mu^{*i}(g)=\int_{G}\cdots\int_{G}f(g_i\cdots g_1)d\mu(g_i)\cdots d\mu(g_1),
\end{align*}
for any function $f:G^i\to\R$.

\subsubsection{Averaging operator}
As previously mentioned, the so-called averaging operator plays a very important role in this work. Denote by $A:L^2(\TX,m)\to L^2(\TX,m)$ the averaging operator with respect to $G$ given by 
\begin{align*}
Af=\int_G f(g x) \,d\mu(g),
\end{align*}
where $f\in L^2(\TX, m)$. We get the $n$'th iterate of $A$ by straight forward calculation, this is
\begin{align*}
A^n f=\int_{G^n} f(g'x)\,d\mu^{*n}(g').
\end{align*}
Since $m$ is $G$-invariant we also get 
\begin{align}
\label{eq:Ainv}
\int_\TX Af\,dm=\int_\TX\int_G f(g x) \,d\mu(g) dm(x)=\int_G \int_\TX f(g x) \,dm(x) d\mu(g)=\int_\TX f\,dm.
\end{align}
Notice that $A$ is linear.
\begin{defn}[]
\label{Def:SpectralGap}
We say that the averaging operator has spectral gap in $L^2(\TX,m)$ if there exists constants $\lambda\in (0,1)$ and $c_0>0$ such that for all $f\in L^2(\TX,m)$ and all $n\in\N$
\begin{align}
\label{eq:SGdef}
\norm{A^n f-\int_\TX f\;dm}_2\leq c_0\lambda^n\norm{f}_2.
\end{align}
\end{defn}

\subsubsection{Distance-like functions}

We are going to introduce two types of function $\Delta$ that we are interested in. For this we need the \textit{tail distribution function} of $\Delta$. We define this as
\begin{align*}
\Phi_{\Delta}(z)=m\para{\tub{x:\Delta(x)\geq z}}\quad,\quad z\in \R.
\end{align*}
Notice that
\begin{align*}
\PP(\xi_0>u_n)=\Phi_{\Delta}(u_n).
\end{align*} 
\begin{defn}[]
\label{DistanceLike}
For $k>0$, we say that $\Delta$ is $k$-DL ("Distance-Like") if it is continuous and satisfies
\begin{align}
\label{eq:def1a}
\exist v_1, v_2>0 \;\text{ such that }\; v_1 e^{-kz}\leq \Phi_{\Delta}(z)\leq v_2 e^{-kz}\enskip,\enskip\foral z\in \R.
\end{align}
For $k>0$, we say that $\Delta$ is $k$-SDL ("Strong-Distance-Like") if it is continuous and satisfies
\begin{align}
\label{eq:def1b}
\exist v_1>0 \;\text{ such that }\; \Phi_{\Delta}(z)=v_1e^{-kz}+o(e^{-kz}) \quad\text{ as } z\to\infty.
\end{align}
\end{defn}
\noindent The notion of distance-like functions was introduced in \cite{KleinMarg}.

Throughout the paper we will make use of big O notation as well as Vinogradov symbols when appropriate. So for a set $S$ and functions $f,g$ on $S$ we write $f(s)=O(g(s))$ if there exists a constant $c$ such that $\num{f(s)}\leq c\num{g(s)}$ for all $s\in S$. We sometimes write $f(s)\ll g(s)$ meaning the same as $f(s)=O(g(s))$ and we write $f(s)\asymp g(s)$ if $f(s)\ll g(s)$ and $g(s)\ll f(s)$.

%
%

\subsection{Bounds on the limiting distribution of $M_n$}

\begin{thm}[]
\label{thm:Res1}
Assume that $\Delta$ is $k$-DL for some $k>0$ and that $A$ has spectral gap on $L^2(\TX,m)$. Set $u_n=r+\frac{1}{k}\log n$. Then for all $r\in\R$
\begin{align*}
e^{-v_2e^{-kr}}\leq\liminf_{n\to\infty} \PP(M_n\leq u_n)\leq \limsup_{n\to\infty} \PP(M_n\leq u_n)\leq e^{\theta_\lambda e^{-kr}},
\end{align*}
for 
\begin{align*}
\theta_{\lambda}=\frac{\lambda}{1-\lambda}c_0v_2-v_1,
\end{align*}
where $v_1,v_2$ and $c_0,\lambda$ are the constants from Definition \ref{DistanceLike} and \ref{Def:SpectralGap} respectively.
\end{thm}
\begin{remark}
We see that for $\lambda$ close to 1, we get $\thelam>0$ rendering the upper bound on the limiting distribution trivial. However, for small values of $\lambda$ we get $\thelam<0$, hence a non-trivial upper bound.
\end{remark}

Naturally, the strategy of the proof will be to verify the assumptions of Theorem \ref{thm:341}. We begin by determining the correct scaling sequences $a_n$ and $b_n$. Assume that $\Delta$ is a $k$-DL function. Then for all $n\in\N$
\begin{align*}
v_1 e^{-ku_n}\leq \Phi_{\Delta}(u_n)\leq v_2 e^{-ku_n},
\end{align*}
for constants $v_1,v_2>0$. Easily, we see that
\begin{align*}
v_1\liminf_{n\to\infty}\para{n e^{-ku_n}}\leq \liminf_{n\to\infty} n\Phi_{\Delta}(u_n)\leq\limsup_{n\to\infty} n\Phi_{\Delta}(u_n) \leq v_2 \limsup_{n\to\infty} \para{n e^{-ku_n}}.
\end{align*}
Since $\Phi_{\Delta}(u_n)=\PP(\xi_0>u_n)$, the upper bound on \eqref{eq:uplowbound2} will be non-trivial if we can find sequences $a_n$ and $b_n$ such that the limit of $ne^{-k(a_n r+b_n)}$ exists and is non-trivial. By writing $ne^{-k(a_n r+b_n)}=ne^{ka_n r}e^{kb_n}$ it is easy to see that for $a_n=1$ and $b_n=\frac1k\log n$ we get
\begin{align*}
\lim_{n\to\infty} ne^{-k(a_n r+b_n)}=e^{-kr}.
\end{align*}
Obviously for this choice of scaling sequences we also get a non-trivial lower bound. We formulate this conclusion as a lemma
\begin{lemma}[]
\label{lem:ScalSeq}
Suppose $\Delta$ is a $k$-DL function and set $u_n=r+\frac1k\log n$. Then
\begin{align*}
v_1e^{-kr}\leq \liminf_{n\to\infty} n\PP(\xi_0>u_n)\leq\limsup_{n\to\infty} n\PP(\xi_0>u_n) \leq v_2e^{-kr},
\end{align*}
where $v_1,v_2>0$ are the constants from Definition \ref{DistanceLike}.
\end{lemma}
\begin{remark}
It follows immediately that if $\Delta$ is assumed to be $k$-SDL, then the lemma holds with the same choice of $u_n$.
\end{remark}
The next lemma verifies Condition $D'_{g(r)}(u_n)$ under the assumptions of Theorem \ref{thm:Res1}. 
\begin{lemma}[]
\label{lem:unD'}
Assume that $\Delta$ is $k$-DL for some $k>0$ and suppose $A$ has spectral gap in $L^2(\TX,m)$. Set $u_n=r+\frac{1}{k}\log n$. Then Condition $D_{g(r)}'(u_n)$ holds for $\xi_i$ with $g(r)=\frac{\lambda}{1-\lambda}c_0v_2e^{-kr}$, where $v_2$ and $c_0,\lambda$ are the constants from Definition \ref{DistanceLike} and \ref{Def:SpectralGap} respectively.
\begin{proof}
We can rewrite the joint probability of $\xi_0$ and $\xi_j$ in terms of integrals of characteristic functions. Set $W:=(u_n,\infty)$, $V_0=\tub{x\in \TX: \Delta(x)\in W}$ and $V_i^{\gb}=\tub{x\in\TX:\xi_i (\gb,x)\in W}$. Notice that $V_0=V_0^{\gb}$ and
\begin{align*}
\een_{V_i^{\gb}}(x)=\een_{V_0}(L^i(\gb)x).
\end{align*}
Then we get
\begin{align*}
\PP(\xi_0>u_n, \xi_j>u_n)&=\int_{G^{\times\N}}\int_\TX \een_{V_0^{\gb}\,\cap\,V_j^{\gb}}(x)\,dm(x)\,d\mu^{\otimes\N}(\gb)\\
&=\int_{G^{\times\N}}\int_\TX \een_{V_0}(x)\een_{V_0}(L^j(\gb)x)\,dm(x)\,d\mu^{\otimes\N}(\gb)\\
&=\int_\TX \een_{V_0}(x)\int_{G^{\times\N}}\een_{V_0}(L^j(\gb)x)\,d\mu^{\otimes\N}(\gb)\,dm(x)\\
&=\int_\TX \een_{V_0}(x)\int_{G^j}\een_{V_0}(g'x)\,d\mu^{*j}(g')\,dm(x)\\
&=\int_\TX \een_{V_0}(x)A^j\para{\een_{V_0}(x)}\,dm(x).
\end{align*}
Set $\psi:=\een_{V_0}$ to get
\begin{align*}
\PP(\xi_0>u_n, \xi_j>u_n)=\int_\TX \psi A^j(\psi)\,dm.
\end{align*}
Recall the Cauchy-Schwartz inequality stating that $\norm{fg}_1\leq \norm{f}_2\norm{g}_2$ for $f,g\in L^2(\TX,m)$. We proceed by estimating the difference $\vert \PP(\xi_0>u_n, \xi_j>u_n)-\PP(\xi_0>u_n)\PP(\xi_j>u_n)\vert$. Written in terms of integrals we have
\begin{gather}
\begin{split}
\label{eq:DeriveLemD'}
\left\vert \int_\TX \psi A^j (\psi)\,dm-\int_\TX \psi \,dm \int_\TX \psi \,dm \right\vert &= \left\vert\int_\TX \psi \para{A^j(\psi)-\int_\TX \psi \,dm} \,dm\right\vert\\
&\leq \int_\TX \left\vert\psi \para{A^j(\psi)-\int_\TX \psi \,dm}\right\vert \,dm\\
&= \norm{\psi \para{A^j(\psi)-\int_\TX \psi \,dm}}_1\\
&\leq \norm{\psi}_2\norm{A^j(\psi)-\int_\TX \psi\, dm}_2\\
&\leq c_0\lambda^j \norm{\psi}_2^2.
\end{split}
\end{gather}
The Cauchy-Schwartz inequality was used to get the second last inequality while the spectral gap property of $A$ was applied to get the final estimate.
It follows that
\begin{align*}
\PP(\xi_0>u_n,\xi_j>u_n)\leq \left(\int_\TX \psi \;dm\right)^2+c_0\lambda^j\norm{\psi}_2^2.
\end{align*}
Since $\psi$ is a characteristic function we know that
\begin{align*}
\norm{\psi}_2^2=\int_\TX \psi^2 \;dm=\int_\TX \psi \;dm.
\end{align*}
We also notice that
\begin{align*}
\int_\TX \psi \;dm=\Phi_{\Delta}(u_n).
\end{align*}
Using that $\Delta$ is $k$-DL and using that $u_n=r+\frac1k \log n$ we see that
\begin{align*}
\PP(\xi_0>u_n,\xi_j>u_n)\leq (v_2e^{-ku_n})^2+c_0\lambda^j v_2e^{-ku_n}=\frac{v_2^2}{n^2}e^{-2kr}+c_0\lambda^j \frac{v_2}{n}e^{-kr}.
\end{align*}
We do the summation from Condition $D'_{g(r)}(u_n)$ to get
\begin{align*}
 n\sum_{j=1}^{\left[\frac{n}{q}\right]}\PP(\xi_0>u_n,\xi_j>u_n)&\leq  n\sum_{j=1}^{\left[\frac{n}{q}\right]}\para{\frac{v_2^2}{n^2}e^{-2kr}+c_0\lambda^j \frac{v_2}{n}e^{-kr}}\\
&\leq \frac{v_2^2}{q}e^{-2kr}+c_0v_2e^{-kr}\sum_{j=1}^{\left[\frac{n}{q}\right]}\lambda^j.
\end{align*}
Recall that since $\lambda\in (0,1)$ we have $\sum_{j=1}^{\infty}\lambda^j=\frac{\lambda}{1-\lambda}$ so when we take the $\limsup_{n\to\infty}$ we get
\begin{align*}
g_q(r)=\limsup_{n\to\infty}n\sum_{j=2}^{\left[\frac{n}{q}\right]}\PP(\xi_0>u_n,\xi_j>u_n)\leq \frac{v_2^2}{q}e^{-2kr}+\frac{\lambda}{1-\lambda}c_0v_2e^{-kr}.
\end{align*}
Finally taking the $\limsup_{q\to\infty}$ gives
\begin{align*}
\limsup_{q\to\infty} g_q(r)\leq \frac{\lambda}{1-\lambda}c_0v_2e^{-kr}.
\end{align*}
So Condition $D'_{g(r)}(u_n)$ holds with $g(r)=\frac{\lambda}{1-\lambda}c_0v_2e^{-kr}$.
\end{proof}
\end{lemma}
\begin{remark}
Notice that $g(r)$ vanishes as the spectral gap $\lambda$ goes to zero.
\end{remark}

\subsubsection{Verifying Condition $D(u_n)$}
\label{SecD}
To verify Condition $D(u_n)$ we need to rewrite the joint distribution function of the $\xi_i$ using the averaging operator. The idea is the same as the one we used to rewrite the joint distribution in the proof of Lemma \ref{lem:unD'}. Now we essentially do the same calculation in higher generality.

Throughout the following computation let $\bar{n}=(n_1,\dots,n_t)$ denote a fixed $t$-tuple of integers where $n_1<\dots<n_t$. Let $W=(-\infty, u_n]$ and again use the notation $V_0=\tub{x\in \TX: \Delta(x)\in W}$ and $V_i^{\gb}=\tub{x\in\TX:\xi_i (\gb,x)\in W}$ introduced in the proof of Lemma \ref{lem:unD'}. Furthermore, set  
\begin{align*}
\Lambda_{\bar{n}}&:=\tub{y\in \TY:\xi_{n_1}(y)\in W,\dots,\xi_{n_t}(y)\in W}\nonumber.
\end{align*}
Using this notation we rewrite the joint distribution function of $\xi_{n_1},\dots,\xi_{n_t}$ in terms of integrals of characteristic functions. 
\begin{align}
\label{JDder}
\PP(\Lambda_{\bar{n}})&=\int_{G^{\times \N}\times \TX} \een_{\Lambda_{\bar{n}}} \,d(\mu^{\otimes\N}\otimes m)\nonumber\\
&=\int_{G^{\times \N}}\int_\TX \prod_{i=1}^t\een_{V_{n_i}^{\gb}}(x) \,dm(x)\,d\mu^{\otimes\N}(\gb)\nonumber\\
&=\int_{G^{\times \N}}\int_\TX \prod_{i=1}^t\een_{V_0}(L^{n_i}(\gb)x)\,dm(x)\,d\mu^{\otimes\N}(\gb)\nonumber\\
&=\int_\TX\int_{G^{\times \N}} \prod_{i=1}^t\psi(L^{n_i}(\gb)x)\,d\mu^{\otimes\N}(\gb)\,dm(x),
\end{align}
where again $\psi:=\een_{V_0}$. It is practical to introduce the notation $g_{[i,j]}=g_i\cdots g_j$ for $i>j$. We now look at the integral with respect to $\mu^{\otimes\N}$ in \eqref{JDder}. We can rewrite this integral using the averaging operator in the following way. First we write
\begin{align}
\label{eq:AvCal1}
\int_{G^{\times \N}} \prod_{i=1}^t\psi(L^{n_i}(\gb)x)\,d\mu^{\otimes\N}(\gb)
&=\int_{G}\cdots\int_{G}\prod_{i=1}^t\psi(g_{n_i}\cdots g_1x)\,d\mu(g_{n_t})\cdots d\mu(g_1).
\end{align}
Now, on the right hand side of \eqref{eq:AvCal1}, look only at the integrals with respect to $g_{n_t},\dots,g_{n_{t-1}+1}$. We get
\begin{align*}
\int_{G}\cdots\int_{G}\prod_{i=1}^t&\psi(g_{n_i}\cdots g_1x)\,d\mu(g_{n_t})\cdots d\mu(g_{n_{t-1}+1})\\
&=\prod_{i=1}^{t-1}\psi(g_{[n_i,1]}x)\int_{G^{n_t-n_{t-1}}}\psi(g_{[n_t,n_{t-1}+1]}g_{[n_{t-1},1]}x)\,d\mu^{*(n_t-n_{t-1})}(g_{[n_t,n_{t-1}+1]})\\
&=\prod_{i=1}^{t-1}\psi(g_{[n_i,1]}x)A^{n_t-n_{t-1}}\psi(g_{[n_{t-1},1]}x).
\end{align*}
Inserting this in \eqref{eq:AvCal1} we get
\begin{align*}
\int_{G^{\times \N}} \prod_{i=1}^t&\psi(L^{n_i}(\gb)x)d\mu^{\otimes\N}(\gb)\\
&=\int_{G}\cdots\int_{G}\prod_{i=1}^{t-1}\psi(g_{[n_i,1]}x)A^{n_t-n_{t-1}}\psi(g_{[n_{t-1},1]}x)d\mu(g_{n_{t-1}})\cdots d\mu(g_1).
\end{align*}
We repeat this step by looking at the integrals in \eqref{eq:AvCal1} with respect to $g_{n_{t-1}},\dots,g_{n_{t-2}+1}$. These integrals, rewritten in terms of the averaging operator as done above, become
\begin{align*}
\prod_{i=1}^{t-2}\psi(g_{[n_i,1]}x)A^{n_{t-1}-n_{t-2}}\para{\psi(g_{[n_{t-2},1]}x)A^{n_t-n_{t-1}}\psi(g_{[n_{t-2},1]}x)}.
\end{align*}
Again we can insert this in \eqref{eq:AvCal1} and repeat the procedure. Doing this $t$ times eventually gives that the integral with respect to $\mu^{\otimes \N}$ in \eqref{JDder} is
\begin{align*}
A^{n_1}\para{\psi(x)A^{n_2-n_1}\para{\psi(x)\dots A^{n_t-n_{t-1}}\para{\psi(x)}}\dots}.
\end{align*}
By integrating again with respect to $m$ and applying \eqref{eq:Ainv} we finally get
\begin{align*}
\PP(\Lambda_{\bar{n}})&=\int_\TX A^{n_1}\para{\psi(x)A^{n_2-n_1}\para{\psi(x)\dots A^{n_t-n_{t-1}}\para{\psi(x)}}\dots}\,dm(x)\\
&=\int_\TX \psi(x)A^{n_2-n_1}\para{\psi(x)\dots A^{n_t-n_{t-1}}\para{\psi(x)}\dots}\,dm(x).
\end{align*}
We can simplify notation by defining the following sequence of operators. For the sequence $\bar{n}=(n_1,\dots,n_t)$ and the fixed function $\psi:=\een_{V_0}$ we define $E_{\bar{n}}^i:L^{\infty}(\TX,m)\to L^{\infty}(\TX,m)$ recursively by
\begin{align*}
&E_{\bar{n}}^1(\phi)=\phi,\\
&E_{\bar{n}}^i(\phi)=E_{\bar{n}}^{i-1}(\psi A^{n_i-n_{i-1}}(\phi)),
\end{align*}
where $\phi\in L^{\infty}(\TX,m)$.
Notice that $E_i$ is linear since $A$ is linear. Using this notation and setting $\phi=\psi$ we get
\begin{align}
\label{Plambda}
\PP(\Lambda_{\bar{n}})=\int_\TX E_{\bar{n}}^{t}(\psi) \,dm.
\end{align}
%
%
%
%

Having rewritten the joint distribution, we proceed by demonstrating how to apply the spectral gap property of the averaging operator. More explicitly, we look at how we can split \eqref{Plambda} into a product of two integrals at the cost of an error term when $A$ has spectral gap. Let $\bar{n}=(n_1,\dots,n_p,n_{p+1},\dots,n_t)$ and also set $\bar{q}=(n_1,\dots,n_p)$ and $\bar{s}=(n_{p+1},\dots,n_t)$. Again assume that $n_1<\dots<n_p<n_{p+1}<\dots<n_t$. We want to estimate the difference
\begin{align*}
\left\vert\PP(\Lambda_{\bar{n}})-\PP(\Lambda_{\bar{q}})\PP(\Lambda_{\bar{s}})\right\vert.
\end{align*}
Written as integrals this is
\begin{align*}
\left\vert \int_\TX E_{\bar{n}}^{t}(\psi)\,dm-\int_\TX E_{\bar{q}}^{p}(\psi)\,dm\int_\TX E_{\bar{s}}^{t-p}(\psi) \,dm\right\vert.
\end{align*}
Notice that
\begin{align*}
E_{\bar{n}}^{t}(\phi)=E_{\bar{q}}^{p}(\psi A^{n_{p+1}-n_p}(E_{\bar{s}}^{t-p}(\phi))).
\end{align*}
Assume now that $A$ has spectral gap in $L^2(\TX,m)$. Let $\sigma:=E_{\bar{s}}^{t-p}(\psi)$. Using the linearity of $E_{\bar{q}}^{p}$ we get 
\begin{gather}
\begin{split}
\label{SGapp}
\bigg\vert \int_\TX E_{\bar{n}}^{t}(\psi)&\,dm-\int_\TX E_{\bar{q}}^{p}(\psi)\,dm\int_\TX E_{\bar{s}}^{t-p}(\psi) \,dm\bigg\vert\\
&=\left\vert \int_\TX E_{\bar{q}}^{p}\left(\psi A^{n_{p+1}-n_p}(\sigma)\right)\; dm-\int_\TX E_{\bar{q}}^{p}\left(\psi \right)\left(\int_\TX \sigma\,dm\right)\,dm\right\vert\\
&=\left\vert \int_\TX E_{\bar{q}}^{p}\left(\psi A^{n_{p+1}-n_p}(\sigma)\right)\; dm-\int_\TX E_{\bar{q}}^{p}\left(\psi \int_\TX \sigma\,dm\right)\,dm\right\vert\\
&=\left\vert \int_\TX E_{\bar{q}}^{p}\left(\psi \left(A^{n_{p+1}-n_p}(\sigma)-\int_\TX \sigma\,dm\right)\right)\, dm \right\vert\\
&\leq\int_\TX \left\vert E_{\bar{q}}^{p}\left(\psi \left(A^{n_{p+1}-n_p}(\sigma)-\int_\TX \sigma\,dm\right)\right)\right\vert\,dm\\
&=\norm{E_{\bar{q}}^{p}\left(\psi \left(A^{n_{p+1}-n_p}(\sigma)-\int_\TX \sigma\,dm\right)\right)}_1.
\end{split}
\end{gather}
Continuing the calculation, set $\rho:=A^{n_{p+1}-n_p}(\sigma)-\int_\TX \sigma\,dm$. Also, recall the H\" older inequality for $L^{\infty}$ functions, that is, for $f\in L^1$ and $g\in L^{\infty}$ we have
\begin{align*}
\norm{fg}_1\leq \norm{f}_1\norm{g}_{\infty}.
\end{align*}
We then get
\begin{gather*}
\begin{split}
\norm{E_{\bar{q}}^{p}(\psi \rho)}_1&=\norm{\psi A^{n_2-n_1}(\dots \psi A^{n_p-n_{p-1}}(\psi \rho))}_1\\
&\leq\norm{\psi}_{\infty}\norm{A^{n_2-n_1}(\psi A^{n_3-n_2}(\dots \psi A^{n_p-n_{p-1}}(\psi \rho))}_1\\
&\leq\norm{\psi}_{\infty}\norm{\psi A^{n_3-n_2}(\dots \psi A^{n_p-n_{p-1}}(\psi \rho))}_1\\
&\,\text{  }\vdots\\
&\leq\norm{\psi}_{\infty}^{p-1}\norm{\psi \rho}_1\\
&\leq \norm{\psi}_{\infty}^{p}\norm{\rho}_1\\
&\leq\norm{\rho}_2.
\end{split}
\end{gather*}
Here we alternated between using the H\" older inequality to split into products of norms and equation \eqref{eq:Ainv} to get rid of the averaging operator. The last inequality holds since $\psi$ is a characteristic function on a probability space. We can now continue the calculation in \eqref{SGapp} by applying the spectral gap property of $A$:
\begin{align*}
\norm{E_{\bar{q}}^{p}(\psi \rho)}_1&=\norm{E_{\bar{q}}^{p}\left(\psi \left(A^{n_{p+1}-n_p}(\sigma)-\int_\TX \sigma\,dm\right)\right)}_1\\
&\leq \norm{\left(A^{n_{p+1}-n_p}(\sigma)-\int_\TX \sigma\,dm\right)}_2\\
&\leq c_0\lambda^{n_{p+1}-n_p}\norm{\sigma}_2\\
&\leq c_0\lambda^{n_{p+1}-n_p},
\end{align*}
since it is easily seen that $\norm{\sigma}_2\leq 1$. All together we have shown that
\begin{align}
\label{eq:SGfinal2}
\left\vert\PP(\Lambda_{\bar{n}})-\PP(\Lambda_{\bar{q}})\PP(\Lambda_{\bar{s}})\right\vert=O\para{\lambda^{n_{p+1}-n_p}}.
\end{align}

%
%

\begin{lemma}[]
\label{lem:D(u_n)}
If $A$ has spectral gap on $L^2(\TX,m)$ then Condition $D(u_n)$ holds for the sequence $\xi_i$ for any choice of scaling sequence $u_n$ and any choice of $\Delta$.
\begin{proof}
Set $W=(-\infty,u_n]$ such that for $\bar{n}=(n_1,\dots,n_t)$ we have
\begin{align*}
\PP(\Lambda_{\bar{n}})=\PP(\xi_{n_1}\leq u_n,\dots,\xi_{n_t}\leq u_n).
\end{align*}
We rewrite the distribution function using the averaging operator as demonstrated earlier. Let $1\leq n_1<\cdots<n_p<n_{p+1}<\cdots<n_t<n n$ be integers such that $n_{p+1}-n_p\geq l$. Set $\bar{n}=(n_1,\dots,n_p,n_{p+1},\dots,n_t)$, $\bar{q}=(n_1,\dots,n_p)$ and $\bar{s}=(n_{p+1},\dots,n_t)$. By definition
\begin{align*}
F_{\xi_{n_1},\dots,\xi_{n_p},\xi_{n_{p+1}},\dots,\xi_{n_t}}(r)=\PP(\xi_{n_1}\leq r,\dots,\xi_{n_p}\leq r, \xi_{n_{p+1}}\leq r\dots,\xi_{n_t}\leq r),
\end{align*}
which means that
\begin{align*}
\left\vert F_{\xi_{n_1},\dots,\xi_{n_p},\xi_{n_{p+1}},\dots,\xi_{n_t}}(u_n)-F_{\xi_{n_1},\dots,\xi_{n_p}}(u_n)F_{\xi_{n_{p+1}},\dots,\xi_{n_t}}(u_n) \right\vert 
\end{align*}
can be written as 
\begin{align}
\label{eq:Dconc}
\left\vert\PP(\Lambda_{\bar{n}})-\PP(\Lambda_{\bar{q}})\PP(\Lambda_{\bar{s}})\right\vert,
\end{align}
which by \eqref{eq:SGfinal2} is bounded by $O\para{\lambda^{n_{p+1}-n_{p}}}$. Since $\lambda^{n_{p+1}-n_{p}}\to 0$ for any sequence $l_n\to\infty$ satisfying $l_n\leq n_{p+1}-n_{p}$, we conclude that Condition $D(u_n)$ holds for any choice of $u_n$ and $\Delta$.
\end{proof}
\end{lemma}

We can now conclude on the proof of Theorem \ref{thm:Res1}. In Lemma \ref{lem:ScalSeq} we determined that the inequalities in \eqref{eq:uplowbound2} are non-trivial for the scaling sequence $u_n=r+\frac1k \log n$ and in Lemma \ref{lem:unD'} we proved that Condition $D'_{g(r)}(u_n)$ is satisfied for $\xi_i$ with $g(r)=\frac{1}{1-\lambda}c_0v_2-v_1$. In Lemma \ref{lem:D(u_n)} we proved that Condition $D(u_n)$ is satisfied for $\xi_i$ for any choice of $u_n$ and $\Delta$. This means that all assumptions of Theorem \ref{thm:341} are satisfied and so Theorem \ref{thm:Res1} follows from Theorem \ref{thm:341}.
%

\begin{cor}
\label{cor:aRes}
For any $a\in\N$, Theorem \ref{thm:Res1} holds with $M_n$ substituted by $M_{n,a}=\max_{0\leq i<n}\xi_{ai}$ and
\begin{align}
\label{eq:a-thetalambda}
\theta_{\lambda}=\frac{\lambda^a}{1-\lambda^a}c_0v_2-v_1
\end{align}
\begin{proof}
Set $\eta_i=\xi_{ai}$ and fix $a\in\N$. First notice that $\xi_i$ being stationary implies that $\eta_i$ is stationary. This also means that the common distribution of $\xi_i$ and $\eta_i$ is the same and so nothing is changed in the proof of Lemma \ref{lem:ScalSeq}. The appropriate scaling sequence for $\eta_i$ is therefore also $u_n=r+\frac1k \log n$.

In Lemma \ref{lem:unD'} replace $j$ by $aj$ throughout the proof to obtain
\begin{align*}
g(r)=v_2e^{-kr}\sum_{j=1}^{\infty} \lambda^{aj}=\frac{\lambda^a}{1-\lambda^a}c_0v_2 e^{-kr}.
\end{align*}
In Lemma \ref{lem:D(u_n)}, equation \eqref{eq:Dconc} is bounded above by $\lambda^{n_{p+1}-n_p}$. The equivalent equation for $\eta_i$ is bounded by $\lambda^{a(n_{p+1}-n_p)}$ and so Condition $D(u_n)$ holds as well. Again all assumptions of Theorem \ref{thm:341} are satisfied and so the corollary follows from Theorem \ref{thm:341}.
\end{proof}
\end{cor}

\subsection{Proving Theorem \ref{thm:[Lattice2]} in the general setting}
\begin{thm}[]
\label{cor:Res2}
Assume that $\Delta$ is $k$-SDL for some $k>0$ and that $A$ has spectral gap on $L^2(\TX,m)$. Let $\tub{m_j}$ be a subsequence in $\N$ such that $\tub{m_{j+1}-m_j}$ is strictly increasing. Also, let $\alpha_n<\beta_n$ denote sequences in $\N$ such that $\alpha_n\to\infty$ and $N_n:=\beta_n-\alpha_n\to\infty$. Then for $u_n=r+\frac1k\log {N_n}$ we have
\begin{align*}
\lim_{n\to\infty}\PP\left(\max_{\alpha_n\leq j\leq \beta_n}(\xi_{m_j})\leq u_n\right)=e^{-v_1e^{-kr}},
\end{align*}
where $v_1>0$ is the constant from Definition \ref{DistanceLike}.
\end{thm}

We first prove a lemma.
\begin{lemma}[]
\label{lem:PLambn}
Suppose $\Delta$ is $k$-SDL for some $k>0$ and let $u_n=r+\frac1k\log n$. Then
\begin{align*}
\lim_{n\to\infty}\PP(\xi_0\leq u_n)^n=e^{-v_1e^{-k r}}.
\end{align*}
\begin{proof}
Notice that $\PP(\xi_0\leq u_n)=1-\Phi_{\Delta}(u_n)$. Using that $\Delta$ is $k$-SDL and $u_n=r+\frac1k\log n$ we get  
\begin{align}
(1-\Phi_{\Delta}(u_n))^n&=\para{1-(n^{-1} v_1e^{-kr})-o(n^{-1})}^n\nonumber\\
&=e^{n\log\para{1-(n^{-1} v_1e^{-kr})-o(n^{-1})}}.\label{eq:6}
\end{align}
We approximate $\log\para{1-(n^{-1} v_1e^{-kr})-o(n^{-1})}$ by its second order Taylor expansion around 0 to get
\begin{align*}
\log\para{1-(n^{-1} v_1e^{-kr})-o(n^{-1})}=-n^{-1} v_1e^{-kr}+o(n^{-1}).
\end{align*}
Inserting this in \eqref{eq:6} we get
\begin{align*}
(1-\Phi_{\Delta}(u_n))^n=e^{v_1e^{-kr}+o(1)}.
\end{align*}
Taking limits gives
\begin{align*}
\lim_{n\to\infty}(1-\Phi_{\Delta}(u_n))^n= e^{-v_1e^{-kr}}.
\end{align*}
\end{proof}
\end{lemma}

%
%
%
%

\begin{proof}[Proof of Theorem \ref{cor:Res2}]
Let $W=(-\infty,u_n]$ and $\bar{m}(n)=(m_{\alpha_n},m_{\alpha_n+1},\dots,m_{\beta_n -1})$. Using the notation from section \ref{SecD} we get that
\begin{align*}
\PP\left(\max_{\alpha_n\leq j< \beta_n}(\xi_{m_j})\leq u_n\right)=\PP(\xi_{m_{\alpha_n}}\leq u_n,\dots,\xi_{m_{\beta_n-1}}\leq u_n)=\PP(\Lambda_{\bar{m}(n)}).
\end{align*}
Set $\bar{q}=(m_{\alpha_n})$ and $\bar{s}=(m_{\alpha_n+1},\dots, m_{\beta_n-1})$. Recall that $\psi=\een_{V_0}$. It then follows from \eqref{eq:SGfinal2} that
\begin{align}
\label{eq:Cor1}
\PP(\Lambda_{\bar{m}(n)})=\para{\int_\TX \psi\,dm}\PP(\Lambda_{\bar{s}})+O\para{\lambda^{m_{\alpha_n+1}-m_{\alpha_n}}}.
\end{align}
Now set $\bar{q}_1=(m_{\alpha_n+1})$ and $\bar{s}_1=(m_{\alpha_n+2},\dots, m_{\beta_n-1})$. We then apply \eqref{eq:SGfinal2} again to get
\begin{align*}
\PP(\Lambda_{\bar{s}})= \para{\int_\TX \psi\,dm}\PP(\Lambda_{\bar{s}_1})+O\para{\lambda^{m_{\alpha_n+2}-m_{\alpha_n+1}}}.
\end{align*}
Inserting this in \eqref{eq:Cor1} while using that $\int_\TX \psi \,dm\leq 1$ gives
\begin{align*}
\PP(\Lambda_{\bar{m}(n)})= \para{\int_\TX \psi\,dm}^2\PP(\Lambda_{\bar{s}_1})+O\para{\lambda^{m_{\alpha_n+2}-m_{\alpha_n+1}}+\lambda^{m_{\alpha_n+1}-m_{\alpha_n}}}.
\end{align*}
Repeating this process $\beta_n-\alpha_n$ times eventually gives
\begin{align}
\label{eq:Cor2}
\PP(\Lambda_{\bar{m}(n)})= \para{\int_\TX \psi\,dm}^{(\beta_n-\alpha_n)} +O\para{\sum_{i=\alpha_n}^{\beta_n-2} \lambda^{m_{i+1}-m_i}}.
\end{align}
Recall the notation $N_n=\beta_n-\alpha_n$ and notice that $\int_\TX \psi \,dm=\PP(\xi_0\leq u_n)$. Since $N_n\to\infty$ for $n\to\infty$ it follows from Lemma \ref{lem:PLambn} that 
\begin{align*}
\lim_{n\to\infty}\para{\int_\TX \psi\,dm}^{N_n}= e^{-v_1e^{-kr}}.
\end{align*}
Also, as $\lambda\in(0,1)$ and $\tub{m_{i+1}-m_i}$ is strictly increasing, we see that 
\begin{align*}
\sum_{i=\alpha_n}^{\beta_n -2}\para{\lambda^{m_{i+1}-m_i}}=O(\lambda^{m_{\alpha_n+1}-m_{\alpha_n}})\to 0,\quad\text{for }\; n\to\infty.
\end{align*}
So taking limits in \eqref{eq:Cor2} gives
\begin{align*}
\lim_{n\to\infty}\PP(\Lambda_{\bar{m}(n)})= e^{-v_1e^{-kr}}.
\end{align*}
\end{proof}

\subsection{Logarithm law for random walks}

\begin{cor}[]
\label{cor:Res3}
Assume that $\Delta$ is $k$-DL for some $k>0$ and that $A$ has spectral gap on $L^2(\TX,m)$.
Then for $\PP$-a.e. $y\in \TY$ we have
\begin{align*}
\limsup_{n\to\infty}\frac{\xi_n(y)}{\log n}=\frac1k.
\end{align*}
\begin{proof}
We prove $\limsup_{n\to\infty}\frac{\xi_n(y)}{\log n}\leq \frac1k$ and $\limsup_{n\to\infty}\frac{\xi_n(y)}{\log n}\geq \frac1k$ for $\PP$-a.e. $y\in \TY$.

The proof of the upper bound is an application of the classical Borel-Cantelli Lemma. For completeness we give the proof. Recall the Borel-Cantelli Lemma stating that for any sequence $A_n\subset\TY$ we have that
\begin{align*}
\sum_{n=1}^{\infty} \PP(A_n)<\infty\;\;\Rightarrow\;\; \PP(\tub{y\in \TY:y\in A_n \text{ for infinitely many }n})=0.
\end{align*}
Let $\epsilon>0$ be given. We look at the sequence of sets
\begin{align*}
A_n=\left\{(\gb,x):\xi_n(\gb,x)\geq \left(\frac1k+\epsilon\right)\log n\right\}.
\end{align*}
Since $\xi_n$ is stationary we have that 
\begin{align*}
\PP(A_n)&=\PP\left(\xi_n\geq \left(\frac1k+\epsilon\right)\log n\right)\\
&=\PP\left(\xi_0\geq \left(\frac1k+\epsilon\right)\log n\right)\\
&=m\left(x:\Delta(x)\geq \left(\frac1k+\epsilon\right)\log n\right)\\
&=\Phi_{\Delta}\left(\left(\frac1k+\epsilon\right)\log n\right).
\end{align*}
Since $\Delta$ is $k$-DL we get
\begin{align*}
\PP(A_n)\leq v_2e^{-k\left(\frac1k+\epsilon\right)\log n}=\frac{v_2}{n^{1+k\epsilon}}
\end{align*}
where $v_2>0$ is the constant from Definition \ref{DistanceLike}. So $\sum_{n=1}^{\infty}\PP(A_n)\leq v_2\sum_{n=1}^{\infty} \frac{1}{n^{1+k\epsilon}}<\infty$ implying that for $\PP$-a.e. $y\in \TY$, the inequality
\begin{align*}
\xi_n(y)\geq \left(\frac1k+\epsilon\right)\log n
\end{align*}
only holds true for finitely many $n$. So by taking the $\limsup_{n\to\infty}$ and dividing by $\log n$ we get
\begin{align*}
\limsup_{n\to\infty}\frac{\xi_n(y)}{\log n}\leq \frac1k+\epsilon.
\end{align*}
Since this holds true for every $\epsilon>0$ we have proved the desired inequality for $\PP$-a.e. $y\in \TY$.

We now prove the lower bound.
Assume for contradiction that the lower bound does not hold, i.e. assume that there exists $\epsilon>0$ such that
\begin{align*}
\PP\left(\limsup_{n\to\infty}\frac{\xi_n}{\log n}\leq \frac1k-\epsilon\right) >\epsilon.
\end{align*}
Let 
\begin{align*}
B=\tub{\limsup_{n\to\infty}\frac{\xi_n}{\log n}\leq \frac1k-\epsilon}
\end{align*}
For each $y\in B$ we can find sufficiently large $n_0\in\N$ such that
\begin{align*}
\sup_{j>n_0}\frac{\xi_{j}}{\log j}\leq \frac1k-\frac{\epsilon}{2}
\end{align*}
This implies that
\begin{align*}
B\subset \bigcup_{n_0\geq 1} \tub{\sup_{j>n_0}\frac{\xi_{j}}{\log j}\leq\frac1k-\frac{\epsilon}{2}}
\end{align*}
Since $\PP(B)>\epsilon$ there must be some $n_1\in\N$ for which
\begin{align}
\label{eq:logdelta}
\PP\para{\tub{\sup_{j>n_1}\frac{\xi_{j}}{\log j}\leq\frac1k-\frac{\epsilon}{2}}}>\delta
\end{align}
for some $\delta>0$.
For any $n_2\geq n_1$ and any $a\in\N$ we have
\begin{align}
\label{eq:Log1}
\delta<\PP\left(\max_{n_1\leq j< n_2}\frac{\xi_j}{\log j}\leq \frac1k-\frac{\epsilon}{2}\right)&\leq\PP\left(\max_{n_1\leq j< n_2}\frac{\xi_j}{\log n_2}\leq \frac1k-\frac{\epsilon}{2}\right)\nonumber\\
&\leq \PP\left(\max_{n_1\leq j< \kpara{\frac{n_2}{a}}}\para{\xi_{aj}}\leq \para{\frac1k-\frac{\epsilon}{2}}\log n_2\right).
\end{align}
We now apply Corollary \ref{cor:aRes} with the goal of obtaining the opposite inequality. For any $a\in\N$ we have
\begin{align*}
\limsup_{n\to\infty} \PP\left(\max_{0\leq j< n}(\xi_{aj})\leq r+\frac1k\log n\right)\leq  e^{\thelam e^{-kr}}.
\end{align*}
where $\thelam=\frac{\lambda^a}{1-\lambda^a}c_0v_2-v_1$. For simplicity we make a change of variables. Set $r=\frac{1}{k}\log s$ where $s\in (0,\infty)$. Then
\begin{align*}
\limsup_{n\to\infty} \PP\left(\max_{0\leq j< n}(\xi_{aj})\leq \frac{1}{k}\log(sn)\right)\leq e^{\thelam s^{-1}}.
\end{align*}
Pick $a\in\N$ sufficiently large to ensure that $\thelam<0$. Let $\delta>0$ be as in \eqref{eq:logdelta}. Then for $s>0$ sufficiently small we get that $e^{\thelam s^{-1}}<\frac{\delta}{2}$. Also by picking $n\in\N$ sufficiently large we get
\begin{align*}
\PP\left(\max_{0\leq j< n}(\xi_{aj})\leq  \frac{1}{k}\log(sn)\right)< e^{\thelam s^{-1}}+\frac{\delta}{2}<\delta.
\end{align*}
Since $\xi_{aj}$ is stationary, we see that
\begin{align*}
\PP\left(\max_{n_1\leq j< n_1+n}(\xi_{aj})\leq  \frac{1}{k}\log(sn)\right)< \delta.
\end{align*}
Since \eqref{eq:Log1} holds for any $n_2\geq n_1$ we can set $n_2:=a(n_1+n)$. Inserting this in \eqref{eq:Log1} gives
\begin{align*}
\PP\left(\max_{n_1\leq j< n_1+n}\para{\xi_{aj}}\leq \para{\frac1k-\frac{\epsilon}{2}}\log (n_1+n)\right)>\delta.
\end{align*} 
Set $n_3:=n_1+n$. It is a simple calculation to show that if we choose $n$ large enough we get
\begin{align*}
\para{\frac1k-\frac{\epsilon}{2}}\log (n_3)< \frac{1}{k}\log(sn).
\end{align*}
This inequality implies the following sequence of inequalities,
\begin{align*}
\delta&<\PP\left(\max_{n_1\leq j< n_3}\para{\xi_{aj}}\leq \para{\frac1k-\frac{\epsilon}{2}}\log (n_3)\right)
\leq \PP\left(\max_{n_1\leq j< n_3}(\xi_{aj})\leq  \frac{1}{k}\log(sn)\right)<\delta,
\end{align*}
which is a contradiction.

%
\end{proof}
\end{cor}


\section{Proofs of main results}
\label{IntroProofs}

At this stage we are almost done with the proofs of the main results concerning maximal excursions and shortest vectors. The only part that remains is to combine the results of the previous section with known results from other papers.

For the closest returns on the torus we still need some additional arguments specific to this setup.

\subsubsection{Proofs of main results for shortest vectors on the space of unimodular lattices}
In the setup of Subsection \ref{SecRes2} it was proven by Kleinbock and Margulis \cite{KleinMarg} (Proposition 7.1) that $\Delta(x)$ as defined in \eqref{eq:ShortestVec} is $d$-SDL. In the proof of the same proposition the explicit value of the constant $w$ is derived as well. Furthermore, we know from Shalom \cite{Sh1} (Theorem C) that in the same setup the averaging operator has spectral gap in $L^2$. Notice that the theorem applies to $\tyk{L}_d$ since we can identify the space with $\drdz$.
So Theorem \ref{thm:[Lattice1]}, Theorem \ref{thm:[Lattice2]} and Corollary \ref{cor:[Lattice]} follow from Corollary \ref{cor:aRes}, Theorem \ref{cor:Res2} and Corollary \ref{cor:Res3} respectively. Using the $d$-SDL property of $\Delta$ and \eqref{eq:a-thetalambda} we see that
\begin{align*}
w(a)=\frac{\lambda^a}{1-\lambda^a}c_0w-w,
\end{align*}
so indeed, $w(a)\to w$ for $a\to\infty$.

\subsubsection{Proofs of main results for maximal excursions on homogeneous spaces}
In the setup of Subsection \ref{SecRes3} it was also proven by Kleinbock and Margulis \cite{KleinMarg} (Proposition 5.1) that $\Delta(x)=\dd(x,x_0)$ is a $k$-DL function for some $k>0$. The spectral gap property of the averaging operator in $L^2$ in this setup also follows from Shalom \cite{Sh1} (Theorem C).
So Theorem \ref{thm:[Cusp]} and Corollary \ref{cor:[Cusp]} follow from Corollary \ref{cor:aRes} and Corollary \ref{cor:Res3} respectively.

%
%

\subsection{Proofs of main results for closest returns on the torus}

We recall the setup of Theorem \ref{thm:[Torus]}. Let $\TX=\T^d$ equipped with Lebesque measure $m$ and Euclidian metric $\dd$. Also, let $G=\textup{Aut}(\T^d)$ equipped with a probability measure $\mu$. Assume that there is no $G_{\mu}$-invariant factor torus $\overline{T}$ such that the projection of $G_{\mu}$ on $\textup{Aut}(\overline{T})$ is amenable. We know from Bekka and Guivarc'h \cite{BG2} (Theorem 5) that the averaging operator has spectral gap in $L^2(\TX,m)$.

Let $x_0\in \TX$ be a fixed point and set $\Delta(x)=-\log\dd(x,x_0)$. The random variables $\xi_i$ are then given by
\begin{align*}
\xi_{i}(\gb,x)=-\log\dd(L^i(\gb)x,x_0).
\end{align*}
The strategy for proving Theorem \ref{thm:[Torus]} is to verify the assumptions of Theorem \ref{thm:341orig}. Notice that Lemma \ref{lem:D(u_n)} verifies Condition $D(u_n)$ for $\xi_i$ with any choice of $u_n$ and $\Delta$.
This means that we are left with the task of determining the scaling sequence $u_n$ such that the limit of $n\PP(\xi_0>u_n)$ is non-trivial and, for this $u_n$, verifing Condition $D'(u_n)$.

First we determine $u_n$. Let $B_r(x_0)\subset \TX$ denote the ball of radius $r$ at $x_0$ and $V_d$ the volume of the unit ball in $\R^d$. Then
\begin{align*}
n\PP(\xi_0>u_n)
&=n\,m\para{\tub{-\log\dd(x,x_0)> u_n}}\\
&=n\,m\para{B_{e^{-u_n}}(x_0)}.
\end{align*}
As in the case of Lemma \ref{lem:ScalSeq} we set $u_n=r+\frac1d\log n$. Since $\TX$ is locally Euclidian we get that for sufficiently large $n$,
%
%
\begin{align*}
n\,m\para{B_{e^{-u_n}}(x_0)}=nV_d e^{-du_n}=V_d e^{-dr},
\end{align*}
and taking limits we get
\begin{align*}
\lim_{n\to\infty} n\PP(\xi_0>u_n)=V_d e^{-dr}.
\end{align*}
Again, we collect this conclusion in a lemma.
\begin{lemma}[]
\label{lem:ScalSeq2}
Set $u_n=r+\frac1d\log n$. Then for the stationary sequence $\xi_i$ we have
\begin{align*}
\lim_{n\to\infty} n\PP(\xi_0>u_n)=V_d e^{-dr}.
\end{align*}
\end{lemma}


Having determined $u_n$, we proceed to verify Condition $D'(u_n)$. This is the step which requires the most work. Fix $\delta\in (0,1)$. Recall the Hardy-Littlewood maximal operator $M$, which for a function $f:\TX\to\C$ is given by
\begin{align*}
Mf(x):=\sup_{R\in (0,\delta)}\frac{1}{m(B_R(x))}\int_{B_R(x)}|f(y)|\,dm(y).
\end{align*}
The Hardy-Littlewood maximal inequality then states that for any $f\in L^1(\TX)$ we have
\begin{align*}
m(\tub{x:Mf(x)>\beta})=O\para{\beta^{-1}\norm{f}_1}
\end{align*}
for every $\beta>0$. This version of the Hardy-Littlewood maximal inequality for functions on $\T^d$ follows easily from the classical version for functions on $\R^d$. Let 
\begin{align*}
E_i=\tub{(\gb,x)\,:\,\dd(L^i(\gb)x,x)<\frac1s}.
\end{align*}
The next lemma gives sufficient assumptions for Condition $D'(u_n)$ to hold.
\begin{lemma}[]
\label{lem:D'}
Suppose that for constants $\alpha\in (0,d)$ and $\kappa>0$ we have that for all $s>0$,
\begin{align}
\label{eq:TorusEst}
\sum_{i=1}^{\kpara{s^{\alpha}}} \;\PP(E_i)= O\para{s^{-\kappa}}.
\end{align}
Then Condition $D'(u_n)$ holds for $\xi_i$ and $u_n(r)=r+\frac1d \log n$ for a.e. $x_0\in \TX$.
\begin{proof}
Using that $\PP=\mu^{\otimes\N}\otimes m$ we can rewrite the estimate as
\begin{align}
\label{eq:est}
\sum_{i=1}^{\kpara{s^{\alpha}}} \int_{G^{\times\N}} m\left(\tub{x\in \TX\,:\,\dd(L^i(\gb)x,x)<\frac1s}\right) d\mu^{\otimes\N}(\gb)\leq \frac{B}{s^{\kappa}}.
\end{align}
Define the function $\Psi_s:\TX\to \R$ by
\begin{align*}
\Psi_s(x)=\sum_{i=1}^{\kpara{s^{\alpha}}} \int_{G^{\times \N}} \een_{\tub{{x\in \TX\,:\,\dd(L^i(\gb)x,x)<\frac1s}}}(x)\, d\mu^{\otimes \N} (\gb),
\end{align*}
and apply the Hardy-Littlewood maximal operator to $\Psi_s$ to get
\begin{align*}
M\Psi_s(x)=\sup_{R\in (0,\delta)}\frac{1}{m(B_R(x))}\int_{B_R(x)}\Psi_s(y)\; dm(y).
\end{align*}
Set $M\Psi_s(x):=\tyk{M}_s(x)$. Using \eqref{eq:est} and the Hardy-Littlewood maximal inequality we get that for every $\beta>0$
\begin{align*}
m\para{\tub{x:\tyk{M}_s(x)>\beta}}= O\para{\beta^{-1}\norm{\Psi_s}_1}=O\para{\beta^{-1} s^{-\kappa}}
\end{align*}
Let $\epsilon>0$. Set $\gamma=\frac{1+2\epsilon}{\kappa}$ and notice that $\gamma\kappa-\epsilon=1+\epsilon>1$. Let $n$ be an integer and substitute $s$ with $n^{\gamma}$ and set $\beta=n^{-\epsilon}$. Then
\begin{align*}
m\para{\tub{x:\tyk{M}_{n^{\gamma}}(x)>n^{-\epsilon}}}=O\para{n^{\epsilon-\gamma\kappa}}.
\end{align*}
Since $\epsilon-\gamma\kappa<-1$ we see that
\begin{align*}
\sum_{n=1}^{\infty}m\para{\tub{x:\tyk{M}_{n^{\gamma}}(x)>n^{-\epsilon}}}= O\para{\sum_{n=1}^{\infty}n^{\epsilon-\gamma\kappa}}<\infty.
\end{align*}
The classical Borel-Cantelli Lemma then tells us that for a.e. $x_0\in \TX$
\begin{align*}
x_0\notin \limsup_{n\to\infty}\tub{x:\tyk{M}_{n^{\gamma}}(x)>n^{-\epsilon}}.
\end{align*}
So there exists a number $N(x_0)$ such that for all $n\geq N(x_0)$ we have $\tyk{M}_{n^{\gamma}}(x_0)\leq n^{-\epsilon}$. That is
\begin{align*}
\tyk{M}_{n^{\gamma}}(x_0)=\sup_{R\in (0,\delta)}\frac{1}{m(B_R(x_0))}\int_{B_R(x_0)}\Psi_{n^{\gamma}}(x) \;dm(x)\leq n^{-\epsilon}.
\end{align*}
Choose $n$ so large that $\frac{1}{n^{\gamma}}\in (0,\delta)$ and set $R=\frac{1}{n^{\gamma}}$. Then
\begin{align*}
\int_{B_{n^{-\gamma}}(x_0)}\Psi_{n^{\gamma}}(x)\; dm(x)\leq \frac{V_d}{n^{\epsilon+\gamma d}}.
\end{align*}
We want to switch back to the real variable $s$ instead of the integer variable $n$ while preserving the inequality above. Let $s\in(n,n+1)$. On the right hand side of the inequality we can clearly substitute $n$ with $s-1$ and the inequality will still hold. The left hand side written out is 
\begin{align*}
\sum_{i=1}^{\kpara{n^{\gamma\alpha}}}\int_{G^{\times \N}}  m\left(B_{n^{-\gamma}}(x_0)\cap \tub{x:\dd(L^i(\gb)x,x)<\frac{1}{n^{\gamma}}}\right) d\mu^{\otimes \N}(\gb).
\end{align*}
We see that by changing $n$ to $s$ inside the integral, the measure of the intersection becomes smaller. However, to ensure that we are not summing over more terms we need to change $n$ to $s-1$ in the upper limit of the sum. All together we get
\begin{align}
\label{eq:HLest2}
\sum_{i=1}^{\kpara{(s-1)^{\gamma\alpha}}}\int_{G^{\times \N}}  m\left(B_{s^{-\gamma}}(x_0)\cap \tub{x:\dd(L^i(\gb)x,x)<\frac{1}{s^{\gamma}}}\right) d\mu^{\otimes \N}(\gb)
\leq \frac{V_d}{(s-1)^{\epsilon+\gamma d}}.
\end{align}
We aim to connect the left hand side of \eqref{eq:HLest2} with the sum in Condition $D'(u_n)$. To do this we derive as follows using the triangle inequality for the inclusion:
\begin{align*}
&\tub{x\in \TX\ssp \xi_0>u_n, \xi_i>u_n}\\
&=\tub{x\in \TX\ssp -\log \dd(x,x_0)>r+\frac1d\log n, \;-\log \dd(L^i(\gb)x,x_0)>r+\frac1d\log n}\\
&=\tub{x\in \TX\ssp \dd(x,x_0)\leq \frac{e^{-r}}{n^{\frac1d}},\; \dd(L^i(\gb) x,x_0)\leq \frac{e^{-r}}{n^{\frac1d}}}\\
&\subset\tub{x\in \TX\ssp\dd(x,x_0)\leq \frac{2e^{-r}}{n^{\frac1d}},\; \dd(L^i(\gb) x,x)\leq \frac{2e^{-r}}{n^{\frac1d}}}\\
&=\tub{x\in \TX\ssp \dd(x,x_0)\leq \frac{1}{l^{\gamma}},\; \dd(L^i(\gb) x,x)\leq \frac{1}{l^{\gamma}}},
\end{align*}
where $l:=\left(\frac{n^{\frac{1}{d}}}{2e^{-r}}\right)^{\frac{1}{\gamma}}$. Notice that the last line is exactly the set inside the integral in \eqref{eq:HLest2} above with $s$ substituted by $l$. Using this gives
\begin{align*}
\sum_{i=1}^{\kpara{(l-1)^{\gamma\alpha}}}&\PP(\xi_0>u_n,\xi_i>u_n)\\
&\leq \sum_{i=1}^{\kpara{(l-1)^{\gamma\alpha}}} \int_{G^{\times \N}} m\left( \dd(x,x_0)\leq \frac{1}{l^{\gamma}}, \dd(L^i(\gb)x,x_0)\leq \frac{1}{l^{\gamma}} \right)\;d\mu^{\otimes \N}(\gb)\\
&\leq \frac{V_d}{(l-1)^{\gamma d+\epsilon}}\\
&\asymp \frac{V_d}{l^{\gamma d+\epsilon}}.
\end{align*}
In the last line above we replaced $l-1$ with $l$ for notational simplicity. We can do this since we are only interested in the behavior as $n\to\infty$. Inserting the expression for $l$ gives
\begin{align*}
\frac{1}{l^{\gamma d+\epsilon}}=
\frac{(2e^{-r})^{(d+\frac{\epsilon}{\gamma})}}{n^{1+\frac{\epsilon}{\gamma d}}}=O\para{n^{-\frac{\epsilon}{\gamma d}-1}}.
\end{align*}
hence
\begin{align}
\label{eq:D'1}
n\sum_{i=1}^{\kpara{(l-1)^{\gamma\alpha}}}\PP(\xi_0>u_n,\xi_i>u_n)=O\para{n^{-\frac{\epsilon}{\gamma d}}}.
\end{align}
Since $(l-1)^{\gamma\alpha}=O(n^{\frac{\alpha}{d}})$ and $\frac{\alpha}{d}<1$, we see that for sufficiently large $n$, $\kpara{(l-1)^{\gamma\alpha}}\leq\kpara{\frac{n}{q}}$ for any $q\in\N$. This means that the left hand side of \eqref{eq:D'1} does not necessarily account for the entire quantity that we need to estimate to verify Condition $D'(u_n)$. To obtain this we need to add
\begin{align*}
n\sum_{i=\kpara{(l-1)^{\gamma\alpha}}+1}^{\kpara{\frac{n}{q}}}\PP(\xi_0>u_n,\xi_i>u_n)
\end{align*}
to the left hand side of equation \eqref{eq:D'1}. To find an upper bound on this sum we apply the averaging operator exactly like in Lemma \ref{lem:unD'}. This gives
\begin{align*}
\sum_{i=\kpara{(l-1)^{\gamma\alpha}}+1}^{\kpara{\frac{n}{q}}}\PP(\xi_0>u_n,\xi_i>u_n)\leq \sum_{i=\kpara{(l-1)^{\gamma\alpha}}+1}^{\kpara{\frac{n}{q}}}\para{\PP(\xi_0>u_n)^2+\PP(\xi_0>u_n)O(\lambda^i)}
\end{align*}
where $\lambda\in(0,1)$ comes from the spectral gap property of the averaging operator. From the proof of Lemma \ref{lem:ScalSeq2} we see that $\PP(\xi_0>u_n)=\frac1n V_d e^{-dr}$. Inserting this gives 
\begin{align*}
\sum_{i=\kpara{(l-1)^{\gamma\alpha}}+1}^{\kpara{\frac{n}{q}}}\PP(\xi_0&>u_n,\xi_i>u_n)\\
&\leq \para{\kpara{\frac{n}{q}}-\kpara{(l-1)^{\gamma\alpha}}}\para{\frac1n V_d e^{-dr}}^2+\frac1n V_d e^{-dr}\sum_{i=\kpara{(l-1)^{\gamma\alpha}}+1}^{\kpara{\frac{n}{q}}} O(\lambda^i)\\
&\leq \frac{n}{q}\para{\frac1n V_de^{-dr}}^2+\frac1n V_d e^{-dr}O(\lambda^{\kpara{(l-1)^{\gamma\alpha}}+1})\\
&= O\para{\frac{1}{qn} + \frac1n\lambda^{\kpara{(l-1)^{\gamma\alpha}}+1}}.
\end{align*}
Consequently,
\begin{align*}
n\sum_{i=\kpara{(l-1)^{\gamma\alpha}}+1}^{\kpara{\frac{n}{q}}} \PP(\xi_0>u_n,\xi_i>u_n)= O\para{q^{-1} + \lambda^{\kpara{(l-1)^{\gamma\alpha}}+1}}.
\end{align*}
Adding this to \eqref{eq:D'1} we get
\begin{align*}
n\sum_{i=1}^{\kpara{\frac{n}{q}}}\PP(\xi_0>u_n,\xi_i>u_n)= O\para{n^{-\frac{\epsilon}{\gamma d}}+q^{-1} + \lambda^{\kpara{(l-1)^{\gamma\alpha}}+1}}.
\end{align*}
Taking the $\limsup$ for $n\to\infty$ gives
\begin{align*}
\limsup_{n\to\infty} n\sum_{i=1}^{\kpara{\frac{n}{q}}}\PP(\xi_0>u_n,\xi_i>u_n)=O\para{q^{-1}},
\end{align*}
and finally by letting $q\to\infty$ we obtain
\begin{align*}
\limsup_{n\to\infty} n\sum_{i=1}^{\kpara{\frac{n}{q}}}\PP(\xi_0>u_n,\xi_i>u_n)\to 0\quad\text{for}\quad q\to\infty.
\end{align*}
We conclude that Condition $D'(u_n)$ has been established.
\end{proof}
\end{lemma}
In the following set $\Omega:=\supp(\mu)$. To complete the proof of Condition $D'(u_n)$ we need to show that the estimate in \eqref{eq:TorusEst} holds for the setup of Theorem \ref{thm:[Torus]}. This is the content of the next lemma.

\begin{lemma}[]
\label{lem:Est}
Assume that there exists $T>1$ such that $\norm{\omega}\leq T$ for all $\omega\in \Omega$. Assume also that $\det(\omega-I)\neq 0$ for all $\omega\in S_{\mu}$. Let $\alpha<d$. Then there exists $\kappa>0$ such that for all $s>0$
\begin{align}
\label{eq:T1}
\sum_{i=1}^{\kpara{s^{\alpha}}} \PP(E_i)=O\para{s^{-\kappa}}.
\end{align}
\begin{proof}


The strategy of the proof is to derive two different upper bounds on
\begin{align*}
\PP(E_i)&=\PP\left(\tub{(\ob,x)\in \Omega^{\times\N}\times \TX\,:\,\dd(L^i(\ob)x,x)<\frac1s}\right),
\end{align*}
using two different methods. One method generates a bound that is good for small values of $i$ while the other method gives a good bound for large values of $i$. Using the two in combination gives the upper bound in \eqref{eq:T1}.
%

\subsubsection{Method 1}
Let $ \ob\in\Omega^{\times \N}$ and for notational simplicity set
\begin{align*}
L^i( \ob):=\omega\in\Omega^i.
\end{align*}
For $s>0$ we look at
\begin{align*}
E_s^{\omega}:=\left\{x\in \TX:\omega x\in B_{\frac1s}(x)\right\}.
\end{align*}
A point $x\in \TX$ can be written as $x=y+\Z^d$ for some $y\in[0,1]^d$. Multiplication by $\omega$ gives
\begin{align*}
\omega x=\omega y+\Z^d.
\end{align*}
Assume that $x\in E_s^{\omega}$. Then
\begin{align}
\label{eq:yset}
\omega y\in y+B_{\frac1s}+\Z^d,
\end{align}
where $B_{\frac1s}$ is the ball of radius $\frac1s$ at 0 in $\R^d$. We see that 
\begin{align*}
m\para{\tub{x\in \TX:\omega x\in\bis}}=vol_{\R^d}\para{\tub{y\in[0,1]^d:\omega y\in y+B_{\frac1s}+\Z^d}}.
\end{align*}
Rearranging \eqref{eq:yset} we get
\begin{align*}
y\in (\omega-I)^{-1}B_{\frac1s}+(\omega-I)^{-1}\Z^d,
\end{align*}
where we used that $\det(\omega-I)\neq 0$. So 
\begin{align*}
m\para{\tub{x\in \TX:\omega x\in\bis}}=m\left(\left[(\omega-I)^{-1}B_{\frac1s}+(\omega-I)^{-1}\Z^d\right]\Big/\Z^d\right).
\end{align*}
We see that $(\omega-I)^{-1}\Z^d$ can at most have finitely many points in $[0,1]^d$ so the measure must be bounded from above by a scalar multiple of $m\para{(\omega-I)^{-1}B_{\frac1s}}$. To estimate the measure we first see that
\begin{align*}
m\para{(\omega-I)^{-1}B_{\frac1s}}&=\frac{1}{\num{\det(\omega-I)}}\;m\para{B_{\frac1s}}.
\end{align*}
Since $\det(\omega-I)\neq 0$ and $\omega$ has integer entries we see that $\num{\det(\omega-I)}\geq 1$. Then
\begin{align*}
m\para{(\omega-I)^{-1}B_{\frac1s}}\leq m\para{B_{\frac1s}}=O\para{s^{-d}}.
\end{align*}
To find an upper bound on the number of copies of $(\omega-I)^{-1}B_{\frac1s}$ in $[0,1]^d$, first notice that
\begin{align*}
(\omega-I)^{-1}=\frac{1}{\det(\omega-I)}\cdot A,
\end{align*}
where $A$ is some integer matrix. This implies that
\begin{align*}
(\omega-I)^{-1}\Z^d/\Z^d\subset\frac{1}{\det(\omega-I)}\Z^d/\Z^d.
\end{align*}
So the integer lattice $\Z^d$ will at most be contracted by the factor $\det(\omega-I)$ in all $d$ directions. This means that
\begin{align*}
\#\para{(\omega-I)^{-1}\Z^d/\Z^d}\leq \num{\det(\omega-I)}^d.
\end{align*}
By assumption $\norm{\omega}\leq T$ for all $\omega\in\Omega$. So for $\omega\in\Omega^i$ it follows simply by multiplying matrices that $\norm{\omega}\leq \para{dT}^i$. Set $\tilde{T}=dT$. By definition of the determinant we then see that $\det(\omega-I)\leq O(\tilde{T}^{di})$ for all $\omega\in  \Omega^i$.
Consequently,
\begin{align*}
\#\para{(\omega-I)^{-1}\Z^d/\Z^d}\leq O(\tilde{T}^{di})^d=O(\tilde{T}^{d^2i}).
\end{align*}
Multiplying the number of sets by the measure of each set we get
\begin{align*}
m\para{E_s^{\omega}}=m\left(\kpara{(\omega-I)^{-1}B_{\frac1s}+(\omega-I)^{-1}\Z^d}\Big/\Z^d\right)= O\para{\frac{\tilde{T}^{d^2i}}{s^d}}.
\end{align*}
Finally, as the upper bound is independent of $\omega=L^i(\ob)$, integrating over $\Omega^{\times \N}$ is trivial and so 
\begin{align*}
\PP(E_i)=\int_{\Omega^{\times \N}} m\left(\tub{x\in \TX\,:\,\dd(L^i( \ob)x,x)<\frac1s}\right) d\mu^{\otimes \N}( \ob)= O\para{\frac{\tilde{T}^{d^2i}}{s^d}}.
\end{align*}

%

\subsubsection{Method 2}

Let again $\ob\in\Omega^{\times \N}$ and $L^i( \ob)=\omega$. Again, for $s>0$ we look at the set 
\begin{align*}
E_s^{\omega}=\tub{x\in \TX\ssp\omega x\in\bis}.
\end{align*}
The idea of how to estimate its measure is to find a set, which contains $E_s^{\omega}$, and whose measure is easier to compute.
Think of $\TX$ as the $d$-cube $[0,1]^d$ and partition this into sub-cubes of the form
\begin{align*}
\bigtimes_{k=1}^d \kpara{\frac{j_k}{s},\frac{j_k+1}{s}},
\end{align*}
where $j_k\in\tub{0,\dots,s-1}$. Clearly, as each vector $\bar{j}=(j_1,\dots,j_d)\in\tub{0,\dots,s-1}^d:=J$ uniquely determines one such sub-cube, we have $s^d$ cubes of volume $\para{\frac{1}{s}}^d$ in the partition. Let $C_{\bar{j}}$ denote the cube corresponding to the vector $\bar{j}$. Clearly
\begin{align*}
\tub{x\in \TX\ssp\omega x\in\bis}=\bigcup_{\bar{j}\in J}\tub{x:x\in C_{\bar{j}},\,\omega x\in \bis}.
\end{align*}
Notice that this is a disjoint union up to measure zero. Let 
\begin{align*}
C_{\bar{j}}^+:=\bigtimes_{k=1}^d \kpara{\frac{j_k}{s}-\frac1s, \frac{j_k+1}{s}+\frac1s}.
\end{align*}
Obviously $C_{\bar{j}}\subset C_{\bar{j}}^+$. We claim that
\begin{align*}
\tub{x:x\in C_{\bar{j}}, \,\omega x\in \bis}\subset \tub{x: x\in C_{\bar{j}}, \,\omega x\in C_{\bar{j}}^+}.
\end{align*}
This is easy to see. Set $x=(x_1,\dots,x_d)$ and $\omega x=(y_1,\dots,y_d)$. If $\omega x\in \bis$, then
\begin{align*}
\dd(\omega x,x)=\sqrt{(x_1-y_1)^2+\cdots +(x_d-y_d)^2}<\frac1s.
\end{align*}
In particular this means that $|x_k-y_k|<\frac1s$ for all $k\in\tub{1,\dots,d}$. Assume further that $x\in C_{\bar{j}}$. Then for every $k$, $x_k\in\kpara{\frac{j_k}{s},\frac{j_k+1}{s}}$ so we must have $y_k\in\kpara{\frac{j_k}{s}-\frac1s, \frac{j_k+1}{s}+\frac1s}$ implying that $\omega x\in C_{\bar{j}}^+$. So we have
\begin{align*}
\tub{x\in \TX\ssp\omega x\in\bis}\subset \bigcup_{\bar{j}\in J}\tub{x:x\in C_{\bar{j}},\,\omega x\in C_{\bar{j}}^+}.
\end{align*}
Taking measures we get
\begin{align*}
m\para{\tub{x\in \TX\ssp\omega x\in\bis}}&\leq m\para{\bigcup_{\bar{j}\in J}\tub{x:x\in C_{\bar{j}},\,\omega x\in C_{\bar{j}}^+}}\\
&\leq\sum_{\bar{j}\in J}m\para{\tub{x:x\in C_{\bar{j}},\,\omega x\in C_{\bar{j}}^+}}\\
&=\sum_{\bar{j}\in J}\int_\TX \een_{\tub{x\ssp x\in C_{\bar{j}},\,\omega x\in C_{\bar{j}}^+}}(x)\,dm(x)\\
&=\sum_{\bar{j}\in J}\int_\TX \een_{C_{\bar{j}}}(x)\een_{C_{\bar{j}}^+}(\omega x)\,dm(x).
\end{align*}
Integrating over $\Omega^{\times \N}$ we get,
\begin{align*}
\int_{\Omega^{\times \N}}m\Big(\Big\{ x\in \TX: L^i( \ob) x\in & \bis\Big\}\Big)d\mu^{\otimes \N}( \ob)\\
&\leq \int_{  \Omega^{\times \N}}\sum_{\bar{j}\in J}\int_\TX \een_{C_{\bar{j}}}(x)\een_{C_{\bar{j}}^+}(L^i( \ob) x)dm(x)d\mu^{\otimes \N}( \ob)\\
&=\sum_{\bar{j}\in J}\int_\TX \een_{C_{\bar{j}}}(x)\int_{  \Omega^{\times \N}}\een_{C_{\bar{j}}^+}(L^i( \ob) x)d\mu^{\otimes \N}( \ob)dm(x)\\
&=\sum_{\bar{j}\in J}\int_\TX \een_{C_{\bar{j}}}(x)A^i\para{\een_{C_{\bar{j}}^+}(x)}dm(x),
\end{align*}
where $A$ is the averaging operator. Performing the analogous calculation as in \eqref{eq:DeriveLemD'} we get
\begin{align*}
&\left \vert \int_\TX \een_{C_{\bar{j}}}A^i\para{\een_{C_{\bar{j}}^+}}dm-\int_\TX \een_{C_{\bar{j}}}dm\int_\TX\para{\een_{C_{\bar{j}}^+}}dm \right \vert\\
&=\num{\int_\TX \een_{C_{\bar{j}}}\para{A^i\para{\een_{C_{\bar{j}}^+}}-\int_\TX\een_{C_{\bar{j}}^+}dm}dm}\\
&\leq\norm{\een_{C_{\bar{j}}}\para{A^i\para{\een_{C_{\bar{j}}^+}}-\int_\TX\een_{C_{\bar{j}}^+}dm}}_1\\
&\leq\norm{\een_{C_{\bar{j}}}}_2\norm{A^i\para{\een_{C_{\bar{j}}^+}}-\int_\TX\een_{C_{\bar{j}}^+}dm}_2\\
&\leq\big\Vert\een_{C_{\bar{j}}}\big\Vert_2 \big\Vert \een_{C_{\bar{j}}^+}\big\Vert_2 O(\lambda^i),
\end{align*}
where $\lambda\in(0,1)$. This gives us
\begin{align*}
\int_\TX \een_{C_{\bar{j}}}A^i\para{\een_{C_{\bar{j}}^+}}dm&\leq \int_\TX \een_{C_{\bar{j}}}dm\int_\TX\para{\een_{C_{\bar{j}}^+}}dm+\big\Vert\een_{C_{\bar{j}}}\big\Vert_2 \big\Vert \een_{C_{\bar{j}}^+}\big\Vert_2O(\lambda^i)\\
&=\para{\frac{3}{s^2}}^{d}+\para{\frac{3}{s^2}}^{\frac{d}{2}}O(\lambda^i).
\end{align*}
Now, recall that there were $s^d$ sub-cubes in the partition of $[0,1]^d$ so instead of summing over all $\bar{j}\in J$, we may multiply by $s^d$ to finally get
\begin{align*}
\PP(E_i)=\int_{\Omega^{\times \N}}m\para{\tub{x\in \TX\ssp L^i( \ob) x\in\bis}}d\mu^{\otimes \N}(\ob)
= O\para{\frac{1}{s^d}+\lambda^i}.
\end{align*}
%
\subsubsection{Combining method 1 and 2}

The idea of combining method 1 and method 2 is that for small values of $i$
\begin{align*}
O\para{\frac{\tilde{T}^{d^2i}}{s^d}},
\end{align*}
is relatively small as $s$ becomes large. Conversely, since $\lambda\in(0,1)$,
\begin{align*}
O\para{\frac{1}{s^d}+\lambda^i}
\end{align*}
is small for large values of $i$ as $s$ grows. Thus if we use the first bound for the first values of $i$ and add the second bound for the last values of $i$ we are optimizing the total upper bound. 

We can write the above idea as
\begin{align*}
\sum_{i=1}^{\kpara{s^{\alpha}}}\PP(E_i) &=\sum_{i=1}^{K}\PP(E_i)+\sum_{i=K+1}^{\kpara{s^{\alpha}}} \PP(E_i)\\
&\ll \sum_{i=1}^{K} \frac{\tilde{T}^{d^2i}}{s^d} + \sum_{i=K+1}^{\kpara{s^{\alpha}}}\para{\frac{1}{s^d}+\lambda^i}.
\end{align*}
for some $K\in\N$. Since $\tilde{T}>1$ we can estimate the first sum by
\begin{align*}
\sum_{i=1}^{K}\frac{\tilde{T}^{d^2i}}{s^d}= O\para{\frac{\tilde{T}^{d^2K}}{s^d}}.
\end{align*}
For the second sum we have
\begin{align*}
\sum_{i=K+1}^{\kpara{s^{\alpha}}} \para{\frac{1}{s^d}+\lambda^i} &\ll\sum_{i=1}^{\kpara{s^{\alpha}}}\para{\frac{1}{s^d}+\lambda^i} - \sum_{i=1}^{K}\para{\frac{1}{s^d}+\lambda^i}\\
&\ll \frac{1}{s^{d-\alpha}}+\lambda^K.
\end{align*}
Choose $K=\delta \log s$ where $\delta>0$ is some constant to be determined. Inserting this we get
\begin{align*}
\sum_{i=1}^{\kpara{s^{\alpha}}}\PP(E_i) &\ll \frac{1}{s^d}\tilde{T}^{d^2(\delta \log s)}+\frac{1}{s^{d-\alpha}}+\lambda^{\delta\log s}\\
&\ll s^{d^2(\delta \log \tilde{T})-d}+ s^{\delta\log \lambda}+s^{\alpha-d}.
\end{align*}
The estimate as a whole must be polynomially decreasing in $s$, so we need all exponents to be negative. This is true for $\alpha-d$ by assumption  and for $\delta\log\lambda$ since $\lambda\in (0,1)$. Also, by choosing $\delta>0$ sufficiently small we get that $d^2(\delta \log \tilde{T})-d<0$. Pick $\delta$ such that this inequality is satisfied and set $\kappa:=\min(|\alpha-d|,|\delta\log\lambda|,|d^2(\delta \log \tilde{T})-d|)$.
%
We then conclude that
\begin{align*}
\sum_{i=1}^{\kpara{s^{\alpha}}}\PP(E_i) = O\para{s^{-\kappa}}.
\end{align*}
\end{proof}
\end{lemma}
We can now conclude on the proof of Theorem \ref{thm:[Torus]}. In Lemma \ref{lem:ScalSeq2} we proved that the correct scaling sequence was $u_n=r+\frac1d\log n$. Lemma \ref{lem:Est} and \ref{lem:D'} together prove that Condition $D'(u_n)$ is satisfied under the assumptions of Theorem \ref{thm:[Torus]}. Condition $D(u_n)$ was proven already in Lemma \ref{lem:D(u_n)}. This means that all assumptions of Theorem \ref{thm:341orig} have been satisfied and so Theorem \ref{thm:[Torus]} follows.

\begin{proof}[Proof of Corollary \ref{cor:[Torus]}]
We want to prove the logarithm law without the assumptions on $\Omega$ and $S_{\mu}$ made in Lemma \ref{lem:Est} hence we cannot apply Theorem \ref{thm:[Torus]} directly. However, the proof of Lemma \ref{lem:unD'} works for the random walk on the torus as well. In this case the role of the $k$-DL assumption is played by the fact that $\PP(\xi_0>u_n)=\frac1n V_d e^{-dr}$ which we derived in the proof of Lemma \ref{lem:ScalSeq2}.

The analogue of Lemma \ref{lem:unD'} for closest returns on the torus implies that the conclusion of Theorem \ref{thm:Res1}, Corollary \ref{cor:aRes} and Corollary \ref{cor:Res3} holds for closest returns on the torus. In particular, Corollary \ref{cor:Res3} then implies Corollary \ref{cor:[Torus]}.     

\end{proof}

\bibliographystyle{maximbib}
\bibliography{BibSpec3}

\end{document}